\RequirePackage{tikz}
\documentclass[a4,10pt]{article}

\usepackage{lmodern}
\usepackage{amssymb,amsmath}

\usepackage{amsthm}
\usepackage{bm}
\usepackage{environ} 
\usepackage{todonotes}
\usepackage{booktabs}
\usepackage{csquotes}
\usepackage{mathtools}
\usepackage{subcaption}
\usepackage{picins}
\usepackage{enumitem}
\usepackage{authblk}
\usepackage[margin=1.2in]{geometry}
\setenumerate{label=\alph*)}
\usepackage[percent]{overpic}
\usepackage[pdftex,colorlinks=true]{hyperref}

\theoremstyle{definition}
\newtheorem{thm}{Theorem}[section]
\newtheorem{defn}[thm]{Definition}

\newtheorem{prop}[thm]{Proposition}
\newtheorem{rem}[thm]{Remark}
\newtheorem{cor}[thm]{Corollary}

\newcommand{\bbf}{\mathbf f}

\newcommand{\bn}{\mathbf n}

\newcommand{\bg}{\mathbf g}
\newcommand{\bA}{\mathbf A}

\newcommand{\bN}{\mathbf N}

\newcommand{\bD}{\mathbf D}

\newcommand{\bI}{\mathbf I}

\newcommand{\bzero}{\mathbf 0}
\newcommand{\bone}{\mathbf 1}
\newcommand{\bv}{\mathbf v}
\newcommand{\by}{\mathbf y}

\newcommand{\bw}{\mathbf w}

\newcommand{\bu}{\mathbf u}

\newcommand{\R}{\mathbb R}
\newcommand{\C}{\mathbb C}

\newcommand{\tm}{\subseteq}
\newcommand{\abs}[1]{\lvert #1\rvert}
\newcommand{\norm}[1]{\lVert #1\rVert}

\newcommand{\qta}{\quad\text{ and }\quad}

\newcommand{\ii}{\mathrm{i}}
\newcommand{\N}{\mathbb{N}}

\renewcommand{\Vec}[1]{\renewcommand*{\arraystretch}{1.2}\begin{pmatrix*}[r]#1\end{pmatrix*}}

\newcommand{\from}{\colon}

\DeclareMathOperator{\diag}{diag}
\DeclareMathOperator{\Span}{span}

\DeclareMathOperator{\re}{Re}

\makeatletter
\newcommand{\vast}{\bBigg@{3}}
\newcommand{\Vast}{\bBigg@{4}}
\makeatother

\makeatletter
\newif\ifcenter@asb@\center@asb@false
\def\center@arstrutbox{%
	\setbox\@arstrutbox\hbox{$\vcenter{\box\@arstrutbox}$}%
}
\newcommand*{\CenteredArraystretch}[1]{%
	\ifcenter@asb@\else
	\pretocmd{\@mkpream}{\center@arstrutbox}{}{}%
	\center@asb@true
	\fi
	\renewcommand{\arraystretch}{#1}%
}
\makeatother

\usepackage{tikz}
\usetikzlibrary{patterns}
\definecolor{colorA}{rgb}{0,0.447,0.741}
\definecolor{colorB}{rgb}{0.85,0.325,0.098}
\definecolor{colorE}{rgb}{0.929,0.694,0.125}
\definecolor{colorF}{rgb}{0.494,0.184,0.556}
\definecolor{colorD}{rgb}{0.466,0.674,0.188}
\definecolor{colorC}{rgb}{0.301,0.745,0.933}
\definecolor{colorG}{rgb}{0.635,0.078,0.184}

\newlength{\tickl}    
\tikzset{axes/.style={thick,-latex}}
\tikzset{lineplot/.style={thick}}
\tikzset{arrow/.style={thick,-latex}} 
\tikzset{thick arrow/.style={ultra thick,-latex}} 
\tikzset{grid lines/.style={very thin,gray!30}}	
\tikzset{point/.style={radius=2pt}}
\tikzset{help line/.style={black,thin,dashed}} 

\makeatletter
\newsavebox{\measure@tikzpicture}
\NewEnviron{scaletikzpicturetowidth}[1]{%
	\def\tikz@width{#1}%
	\begin{lrbox}{\measure@tikzpicture}%
		\BODY
	\end{lrbox}%
	\pgfmathparse{#1/\wd\measure@tikzpicture}%
	\BODY
}
\makeatother

\begin{document}

\title{On the stability of strong-stability-preserving modified Patankar Runge-Kutta schemes}
\author[1]{Juntao Huang}
\author[2]{Thomas Izgin} 
\author[2]{Stefan Kopecz}
\author[2]{Andreas Meister}
\author[3]{Chi-Wang Shu}

\affil[1]{Department of Mathematics, Michigan State University, East Lansing, MI 48824, USA}
\affil[1]{huangj75@msu.edu}
\affil[2]{Department of Mathematics and Natural Sciences, University of Kassel, Germany}
\affil[2]{izgin@mathematik.uni-kassel.de\ \&\ kopecz@mathematik.uni-kassel.de\ \&\ meister@mathematik.uni-kassel.de}
\affil[3]{Division of Applied Mathematics, Brown University, Providence, RI 02912, USA}
\affil[3]{chi-wang\_shu@brown.edu}
\setcounter{Maxaffil}{0}
\renewcommand\Affilfont{\itshape\small}
\date{}
	\maketitle
	\begin{abstract} 
		In this paper, we perform stability analysis for a class of second and third order accurate
		strong-stability-preserving modified Patankar Runge-Kutta (SSPMPRK) schemes,
		which were introduced in \cite{MR3934688,MR3969000} and can be used to solve convection equations with stiff source terms, such as reactive
		Euler equations, with guaranteed positivity under the standard CFL condition due to the
		convection terms only.
		The analysis allows us to identify
		the range of free parameters in these SSPMPRK schemes in order to ensure stability.
		Numerical experiments are provided to demonstrate the validity of the analysis.
	\end{abstract}

	%
	%

	\section{Introduction}\label{sec:intro}

Recently, structure-preserving numerical methods have attracted much attention due to many successful applications. 
The modified Patankar Runge-Kutta (MPRK) method was firstly introduced in \cite{BDM2003} and preserves the positivity and conservativity of the numerical solution of positive and conservative production-destruction systems (PDS). The PDS have the following form:
\begin{equation}\label{eq:proDesODEs}
	y'_i = P_i(\by) - D_i(\by), \quad i = 1,\dotsc, N
\end{equation}
with
\begin{equation}
	P_i(\by) = \sum_{j=1}^N p_{ij}(\by), \quad 	D_i(\by) = \sum_{j=1}^N d_{ij}(\by)
\end{equation}
and 
\begin{equation}
	p_{ij}(\by) = d_{ji}(\by) \ge 0.
\end{equation}
Here $\by=(y_1, \dotsc, y_N)^T$ and $y_i=y_i(t)> 0$ denotes the concentration of the $i$-th component. The production function $p_{ij}(\by)$ denotes the rate at which the $j$-th component transforms into the $i$-th component, while the destruction function $d_{ij}(\by)$ denotes the rate at which the $i$-th component transforms into the $j$-th component. The exact solutions of \eqref{eq:proDesODEs} share the conservation property, i.e., $\sum_{i=1}^N y_i(t)$ remains unchanged with respect to time $t$. Also, the positivity of the solution is guaranteed as long as the initial condition is positive and $d_{ij}(\by)=0$ for $y_i =0$ \cite{BDM2003}.

There has been a considerable interest in the development of MPRK schemes in recent years. In\cite{KM18,KM18Order3}, MPRK schemes of second and third order were introduced. In \cite{MR4064785}, the authors adapted the modified Patankar trick to deferred correction (DeC) schemes and developed MPDeC schemes of arbitrary order of accuracy. 
In \cite{MR3934688,MR3969000}, instead of using the Runge-Kutta (RK) schemes in the classical form, the authors rewrote the RK 
schemes in the Shu-Osher form \cite{shu1988efficient} and developed another class of MPRK schemes 
for \eqref{eq:proDesODEs}, the so-called strong-stability-preserving MPRK (SSPMPRK) schemes. The purpose was to match the treatment of the convection terms and the
reaction terms in the same RK framework.  This framework was then applied to semi-discrete schemes arising from 
multispecies reactive Euler equations, in which the convection parts were treated with the SSPRK method \cite{gottlieb2001strong} and the stiff reactive source terms were treated with the MPRK method. 
Combining with the  finite-difference WENO schemes, the positivity-preserving WENO scheme was obtained. It is notable 
that, to guarantee the positivity of the numerical solution, the time step size was only constrained by the maximum 
characteristic speeds of the convection term and was independent of the stiffness of the reactive sources.
Accuracy and positivity-preservation were analyzed in \cite{MR3934688,MR3969000}.

In this paper, following the lines of \cite{IKM22Sys}, we investigate the stability behavior of the above mentioned numerical methods in \cite{MR3934688,MR3969000} applied to the stable linear positive and conservative PDS of the form
	\begin{equation}
		\by'(t)=\bA\by(t) \label{eq:Dahlquist_System}
	\end{equation}
	with $\bA\in \R^{N\times N}$ 
	and the initial condition
	\begin{equation}
		\by(0)=\by^0>\bzero \label{eq:IC}.
	\end{equation}
	 The presence of exactly $k>0$ linear invariants means that there exist vectors $\bn_1,\dotsc,\bn_k$ which form a basis of $\ker(\bA^T)$, and hence, satisfy that $\bn_i^T\by(t)=\bn_i^T\by^0$ for all $t\geq 0$ and $i=1,\dotsc, k$.
	The conservativity means that $\bone\in\ker(\bA^T)$.
	In addition, the system \eqref{eq:Dahlquist_System} is positive if and only if the matrix $\bA$ is a Metzler matrix, i.\,e.\ a matrix with nonnegative off-diagonal elements, see \cite{Luen79}, which guarantees $\by(t)>\bzero$ for all $t>0$ whenever $\by^0>\bzero$. 
	Moreover, to ensure stable steady states $\by^*\in \ker(\bA)$,
	the matrix $\bA$ in \eqref{eq:Dahlquist_System} must have a spectrum $\sigma(\bA)\tm \C^-=\{z\in \C\mid \operatorname{Re}(z)\leq0\}$ and the eigenvalues of $\bA$ with vanishing real part have to be associated with a Jordan block size of 1, see \cite[Theorem 3.23]{MR1912409}. Indeed, the same theorem states that no steady state of the test equation \eqref{eq:Dahlquist_System} is asymptotically stable.
	
The paper is organized as follows. In Section \ref{sec:stab_dyn_sys}, we summarize important definitions and results concerned with the stability of fixed points. In Section \ref{sec:stab_SSPMPRK}, we apply Theorem \ref{Thm:_Asym_und_Instabil} and Theorem \ref{Thm_MPRK_stabil} to investigate the stability of the second and third order SSPMPRK schemes from \cite{MR3934688, MR3969000}. Numerical experiments are presented in Section \ref{sec:Num_Tests} to validate the theoretical analysis. Finally, we collect our conclusions and future research topics in Section \ref{sec:Summary}.

	\section{Stability Theory}\label{sec:stab_dyn_sys}

In the following, we summarize the main results on the stability of fixed points of a potentially nonlinear mapping $\bg\colon D\to D$ with $D\tm \R^N$. Thereby, we use $\norm{\ \cdot\  }$ to represent an arbitrary norm in $\R^l$ for $l\in \N$ and $\bD\bg$ denotes the Jacobian of the map $\bg$. 
%
\begin{defn}\label{Def_Lyapunov_Diskr}
	Let $\by^*$ be a fixed point of an iteration scheme $\by^{n+1}=\bg(\by^n)$, that is $\by^*=\bg(\by^*)$. 
	\begin{enumerate}
		\item\label{def:stab} $\by^*$ is called \emph{Lyapunov stable} if, for any $\epsilon>0$, there exists a $\delta=\delta(\epsilon)>0$ such that $\norm{\by^0-\by^*}<\delta$ implies $\norm{\by^n- \by^*}<\epsilon$ for all $n\geq 0$.
		\item If in addition to a), there exists a constant $c>0$ such that $\Vert \by^0-\by^*\Vert<c$ implies $\Vert \by^n-\by^*\Vert \to 0$ for $n\to \infty$, we call $\by^*$ \emph{asymptotically stable.}
		\item A fixed point that is not Lyapunov stable is said to be \emph{unstable}.
	\end{enumerate}
\end{defn}
In the following, we will also briefly speak of stability instead of Lyapunov stability. The next theorem gives sufficient conditions for the analysis of a general $\mathcal C^1$-map $\bg$ based on its Jacobian.
\begin{thm}[{\cite[Theorem 1.3.7]{SH98}}]\label{Thm:_Asym_und_Instabil}
	Let  $\by^{n+1}=\bg(\by^n)$ be an iteration scheme with fixed point $\by^*$. Suppose the Jacobian $\bD\bg(\by^*)$ exists and denote its spectral radius by $\rho(\bD\bg(\by^*))$. Then
	\begin{enumerate}
		\item $\by^*$ is asymptotically stable if $\rho(\bD\bg(\by^*))<1$. 
		\item $\by^*$ is unstable if $\rho(\bD\bg(\by^*))>1$.
	\end{enumerate}
\end{thm}
Unfortunately, the above theorem does not provide the condition to conclude the stability of a fixed point that is not asymptotically stable. However, as mentioned in the introduction, we consider the linear system \eqref{eq:Dahlquist_System} that possesses only stable but not asymptotically stable steady states. As a numerical scheme should transfer the properties of steady states to those of the corresponding fixed points of the method, we are concerned with stable but not asymptotically stable fixed points. According to Theorem \ref{Thm:_Asym_und_Instabil}, such a fixed point $\by^*$ of a $\mathcal C^1$-map $\bg$ must correspond to the case $\rho(\bD\bg(\by^*))=1$, from which it follows that $\by^*$ is non-hyperbolic. The analysis of non-hyperbolic fixed points is more delicate since the analysis of the linearized method is generally not enough to understand the stability properties of the fixed point. However, under certain circumstances, the stability of a non-hyperbolic fixed point still can be determined by means of a linearization as the result from \cite{IKM22Sys} states.

To formulate the theorem, we introduce the matrix
\begin{equation}\label{eq:N}
	\bN=\begin{pmatrix}
		\bn_1^T\\
		\vdots\\
		\bn_k^T
	\end{pmatrix}\in \R^{k\times N} ,
\end{equation}
where $\bn_1,\dotsc,\bn_k$ form a basis of $\ker(\bA^T)$ with $\bA$ from \eqref{eq:Dahlquist_System} and define
\begin{equation}
	H=\{\by\in \R^N\mid \bN\by=\bN\by^*\}\label{eq:H}.
\end{equation}
We point out that under the assumption $\by\in H\cap D$ we obtain $\bg(\by)\in H\cap D$, if and only if $\bg$ conserves all linear invariants. 
\begin{thm}[{\cite[Theorem 2.9]{IKM22Sys}}]\label{Thm_MPRK_stabil}
	Let $\bA\in \R^{N\times N}$ such that $\ker(\bA)=\Span(\bv_1,\dotsc,\bv_k)$ represents a $k$-dimensional subspace of $\R^N$ with $k>0$. Also, let $\by^*\in \ker(\bA)$ be a fixed point of $\bg\from D\to D$ where $D\tm \R^N$ contains a neighborhood $\mathcal D$ of $\by^*$. Moreover, let any element of $C=\ker(\bA)\cap \mathcal D$ be a fixed point of $\bg$ and suppose that $\bg\big|_\mathcal{D}\in \mathcal C^1$ as well as that the first derivatives of $\bg$ are Lipschitz continuous on $\mathcal{D}$. Then $\bD\bg(\by^*)\bv_i=\bv_i$ for $i=1,\dotsc, k$ and the following statements hold:
	\begin{enumerate}
		\item\label{it:Thma} If the remaining $(N-k)$ eigenvalues of $\bD\bg(\by^*)$ have absolute values smaller than $1$, then $\by^*$ is stable.\label{It:Thm_Stab_a}
		\item\label{it:Thmb} Let $H$ be defined by \eqref{eq:H} and $\bg$ conserve all linear invariants, which means that $\bg(\by)\in H\cap D$ for all $\by\in H\cap D$. If additionally the assumption of \ref{It:Thm_Stab_a} is fulfilled, then there exists a $\delta>0$ such that $\by^0\in H\cap D$ and $\norm{\by^0-\by^*}<\delta$ imply $\by^n\to \by^*$ as $n\to \infty$.
	\end{enumerate}
\end{thm}
We would like to mention that the second part of the above theorem does not imply that the fixed point is asymptotically stable but rather attracting for appropriately chosen starting vectors $\by^0$. Indeed, in the situation of Theorem \ref{Thm_MPRK_stabil} no $\by^*\in \ker(\bA)\cap D$ is asymptotically stable as in any neighborhood of $\by^*$ there exist infinitely many other fixed points. 
\begin{rem}\label{rem:C2->C1}
As a final remark, we note that if $\bg\in \mathcal{C}^2$, then we may choose $\mathcal D\tm D$ in such a way that $\overline{\mathcal D}\tm D$. As a result the second derivatives are bounded on the compact set $\bar{\mathcal D}$, so that the first derivatives are Lipschitz continuous due to the mean value theorem. Therefore, $\bg$ restricted to $\mathcal D$ is a $\mathcal C^1$-map with Lipschitz continuous derivatives.
\end{rem}

	\section{Stability of SSPMPRK schemes}\label{sec:stab_SSPMPRK}
	As SSPMPRK schemes can only be directly applied to positive and conservative PDS, we assume that the linear test equation \eqref{eq:Dahlquist_System} is conservative (i.e., $\bone\in\ker(\bA^T)$) and the system matrix $\bA$ is a Metzler matrix. Then the test equation can be rewritten as a positive and conservative PDS with $p_{ij}(\by) = d_{ji}(\by) = a_{ij}y_j$ for $i\neq j$ and $p_{ii}=d_{ii}=0$. 
Moreover, from $\bone\in\ker(\bA^T)$, one can easily derive $\sum_{j=1}^N a_{ji} = 0$ and thus obtain
\begin{equation}\label{eq:-sum(dij)}
-\sum_{\substack{j=1}}^Nd_{ij}(\by) = -\sum_{\substack{j=1\\j\neq i}}^Na_{ji}y_i = a_{ii}y_i,
\end{equation}
which will be used in the following to write the SSPMPRK schemes in the matrix-vector notation.
\subsection{Second order SSPMPRK Scheme}
The second order SSPMPRK scheme for solving \eqref{eq:proDesODEs}, introduced in \cite{MR3934688}, is given by
\begin{equation}\label{eq:SSPMPRK2}
\begin{aligned}
y_i^{(1)}={}& y_i^n+\beta \Delta t\left(\sum_{j=1}^Np_{ij}(\by^n)\frac{y_j^{(1)}}{y_j^n}- \sum_{j=1}^Nd_{ij}(\by^n)\frac{y_i^{(1)}}{y_i^n}\right),\\
y_i^{n+1}={}& (1-\alpha)y_i^n+\alpha y_i^{(1)}\\&+\Delta t\Biggl(\sum_{j=1}^N\left(\beta_{20}p_{ij}(\by^n)+\beta_{21}p_{ij}(\by^{(1)})\right)\frac{y_j^{n+1}}{(y_j^n)^{1-s}(y_j^{(1)})^s}  \\&- \sum_{j=1}^N\left(\beta_{20}d_{ij}(\by^n)+\beta_{21}d_{ij}(\by^{(1)})\right)\frac{y_i^{n+1}}{(y_i^n)^{1-s}(y_i^{(1)})^s}\Biggr),
\end{aligned}
\end{equation}
where $\beta_{20}=1-\frac{1}{2\beta}-\alpha\beta$, $\beta_{21}=\frac{1}{2\beta}$ and $s=\frac{1-\alpha\beta+\alpha\beta^2}{\beta(1-\alpha\beta)}$. Thereby, the free parameters $\alpha$ and $\beta$ satisfy
\begin{equation}
0\leq \alpha\leq 1,\quad \beta>0,\quad \alpha\beta+\frac{1}{2\beta}\leq 1.\label{eq:alphabeta_conditions}
\end{equation}
We refer to the above scheme as SSPMPRK2($\alpha,\beta$).
When applied to \eqref{eq:Dahlquist_System}, the terms $p_{ij}$ and $d_{ij}$ fulfill \eqref{eq:-sum(dij)}. 
As a consequence, the scheme \eqref{eq:SSPMPRK2} can be rewritten as
\begin{equation}\label{eq:SSPMPRK2Matrixvector}
\begin{aligned}
\bzero=&\bm \Phi_1(\by^n,\by^{(1)}) = \by^n+\beta \Delta t\bA\by^{(1)}-\by^{(1)},\\
\bzero=&\bm \Phi_{n+1}(\by^n,\by^{(1)},\by^{n+1}) = (1-\alpha)\by^n+\alpha\by^{(1)}\\&+\Delta t\bA\diag(\by^{n+1})(\diag(\by^{(1)}))^{-s}(\diag(\by^{n}))^{s-1}(\beta_{20}\by^n+\beta_{21}\by^{(1)})-\by^{n+1},\\
\end{aligned}
\end{equation}
where we use the notation $(\diag(\by))_{ij}=\delta_{ij}y_i$ with the Kronecker delta $\delta_{ij}$ as well as $((\diag(\by))^{x})_{ij}=\delta_{ij}y_i^{x}$ for $x\in\R$.  Furthermore, $\by^{(1)}=\by^{(1)}(\by^n)$ and $\by^{n+1}=\bg(\by^n)$ defined by \eqref{eq:SSPMPRK2Matrixvector} are functions of $\by^n$. In order to apply Theorem \ref{Thm:_Asym_und_Instabil} and Theorem \ref{Thm_MPRK_stabil}, we have to investigate the map $\bg$ with respect to its smoothness as well as steady state and linear invariants preservation. 

First of all, we show that $\bg\in \mathcal C^2$ and then use Remark \ref{rem:C2->C1} in order to see that the first derivatives are Lipschitz continuous on an appropriately chosen neighborhood $\mathcal D$ of $\by^*$. Indeed, the maps $\bm \Phi_1\colon\R^N_{>0}\times \R^N_{>0}\to \R^N$ and $\bm \Phi_{n+1}\colon\R^N_{>0}\times \R^N_{>0}\times \R^N_{>0}\to \R^N$ are in $\mathcal C^2$, and as defined in \eqref{eq:SSPMPRK2Matrixvector}, vanish for the argument $(\by^n,\by^{(1)}(\by^n))$ and $(\by^n,\by^{(1)}(\by^n),\bg(\by^n))$, respectively. And since the computation of $\by^{n+1}$ requires only the solution of linear systems which possess always a unique solution for any $\by^n>\bzero$, the function $\bg$ is also a $C^2$-map.

Next, we show that any positive steady state of \eqref{eq:Dahlquist_System} is a fixed point of $\bg$. To see this, we want to mention that $\by^n=\by^{(1)}=\by^{n+1}=\by^*$ is a solution to the system of equations \eqref{eq:SSPMPRK2Matrixvector} due to $\bA\by^*=\bzero$. Since the solution for given $\by^n$ is unique, we conclude that $\by^n=\by^*$ implies $\by^{(1)}=\by^{n+1}=\by^*$, i.\,e.\ $\bg(\by^*)=\by^*$.
 
Moreover, $\bg$ conserves all linear invariants since $\bn^T\bA=\bzero$ and \eqref{eq:SSPMPRK2Matrixvector} imply
 \begin{equation*}
 	\bn^T\bg(\by^n)=\bn^T\by^{n+1}=(1-\alpha)\bn^T\by^n+\alpha\bn^T\by^{(1)}+\bzero=(1-\alpha)\bn^T\by^n+\alpha\bn^T(\by^{n}+\beta \Delta t\bA \by^{(1)})=\bn^T\by^n.
 \end{equation*}
Therefore, the map $\bg\colon \R^N_{>0}\to \R^N_{>0}$ meets the assumptions of Theorem \ref{Thm_MPRK_stabil}, so that we now focus on computing the Jacobian $\bD\bg(\by^*)$. Instead of calculating $\bg$ explicitly, we take advantage of the fact that $\bg(\by^n)=\by^{n+1}$ occurs as an argument within the function $\bm\Phi_{n+1}$, we prefer to compute its total derivative $\bD \bm\Phi_{n+1}$ and solve for $\bD\bg(\by^*)$. Doing so, we have to compute several partial derivatives. The Jacobian is obtained by 
differentiating with respect to the first argument, and we use the notation $\bD_n$  for the corresponding operator as 
we plug in $\by^n$ in \eqref{eq:SSPMPRK2Matrixvector} as the first variable. Similarly, we introduce the operators 
$\bD_{1}$ for the derivatives with respect to the variable where we plugged in $\by^{(1)}$, and similarly, the 
operator $\bD_{n+1}$. As we are interested in the Jacobian of $\bg$ evaluated at $\by^n=\by^*$, we plug in the 
values $(\by^*,\by^{(1)}(\by^*))$ or  $(\by^n(\by^*),\by^{(1)}(\by^*),\bg(\by^*))$, respectively. To indicate this in 
the following formula, we use the notation $\bD_n^*$, $\bD_1^*$ and $\bD_{n+1}^*$. Therefore, we obtain
\begin{equation*}
	\bzero=\bD^*\bm\Phi_{n+1}=	\bD^*_n\bm \Phi_{n+1}+ \bD^*_1\bm \Phi_{n+1}\bD^*\by^{(1)}+\bD^*_{n+1}\bm \Phi_{n+1}\bD\bg{(\by^*)},
\end{equation*}
where $\bD^*\by^{(1)}$ is the Jacobian of $\by^{(1)}(\by^n)$ evaluated at $\by^*$.
If $\bD^*_{n+1}\bm \Phi_{n+1}$ is nonsingular, we can solve for $\bD\bg(\by^*)$ which results in
\begin{equation}\label{eq:Formula_Dg(y*)}
		\bD\bg(\by^*)=-(\bD^*_{n+1}\bm \Phi_{n+1})^{-1}\left(	\bD^*_n\bm \Phi_{n+1}+ \bD^*_1\bm \Phi_{n+1}\bD^*\by^{(1)}\right).
\end{equation}
In order to compute $\bD^*\by^{(1)}$, we use the same trick, i.\,e.\
\begin{equation*}
	\bzero=\bD^*\bm\Phi_1=\bD_n^*\bm\Phi_1+\bD_1^*\bm\Phi_1\bD^*\by^{(1)},
\end{equation*}
which yields
\begin{equation}\label{eq:Formula_D*y^1}
	\bD^*\by^{(1)}=-\left(\bD_1^*\bm\Phi_1\right)^{-1}\bD_n^*\bm\Phi_1,
\end{equation}
if $\bD_1^*\bm \Phi_1$ is invertible. Hence, we have to compute several Jacobians in order to calculate $\bD\bg(\by^*)$ and we start with
\begin{equation*}
	\begin{aligned}
		\bD^*_n\bm \Phi_1=\bI \qta \bD^*_1\bm \Phi_1=\beta \Delta t\bA-\bI.
	\end{aligned}
\end{equation*}
Note that $\beta>0$ and $\sigma(\bA)\tm \C^-$, which implies that $\bD_1^*\bm \Phi_1$ is nonsingular. Thus, we can use \eqref{eq:Formula_D*y^1} and find
\begin{align}\label{eq:D*y^1}
\bD^*\by^{(1)}=-(\beta \Delta t\bA-\bI)^{-1}\cdot \bI=(\bI-\beta \Delta t\bA)^{-1}.
\end{align}
Next, we compute $\bD^*_n\bm\Phi_{n+1}$ and $\bD^*_1\bm\Phi_{n+1}$. For this, we first define $\bbf(\by^n,\by^{(1)})=\diag(\by^n)^k(\beta_{20}\by^n+\beta_{21}\by^{(1)})$ for some $k\in \R$ and get
\begin{equation}\label{eq:ProductRuleDiag}
	\begin{aligned}
\left(\bD_n^*\bbf\right)_{ij}&= \partial_{y_j^n}\left((y_i^n)^k(\beta_{20}y_i^n+\beta_{21}y_i^{(1)}))\right)\Big|_{\by^n=\by^*}\\&=\delta_{ij}\left(k(y_i^*)^{k-1}(\beta_{20}+\beta_{21})y_i^*+(y_i^*)^k\beta_{20}\right)\\&=\left(\diag(\by^*)^{k}\right)_{ij}(k(\beta_{20}+\beta_{21})+\beta_{20}),
	\end{aligned}
\end{equation}
where we have used the fact that $\by^{(1)}(\by^*)=\by^*$. Similarly, defining $\bu(\by^n,\by^{(1)})=\diag(\by^{(1)})^k(\beta_{20}\by^n+\beta_{21}\by^{(1)})$, we obtain
\begin{equation}
\bD_1^*\bu=\diag(\by^*)^{k}(k(\beta_{20}+\beta_{21})+\beta_{21}).\label{eq:D_1u}
\end{equation}
In order to apply the formulas \eqref{eq:ProductRuleDiag} and \eqref{eq:D_1u} to compute  $\bD^*_n\bm\Phi_{n+1}$ and $\bD^*_1\bm\Phi_{n+1}$, we also make use of the fact that diagonal matrices commute, so that we end up with
\begin{align}
	\bD^*_n\bm \Phi_{n+1}=(1-\alpha)\bI+\Delta t\bA((s-1)(1-\alpha\beta)+\beta_{20})\qta \bD^*_1\bm \Phi_{n+1}=\alpha \bI+\Delta t\bA(-s(1-\alpha \beta)+\beta_{21}),
\end{align}
where we have exploited $\beta_{20}+\beta_{21}=1-\alpha\beta$.
Finally, to compute $\bD_{n+1}^*\bm\Phi_{n+1}$ we rewrite \eqref{eq:SSPMPRK2Matrixvector} utilizing $\diag(\bv)\bw=\diag(\bw)\bv$ to get
\begin{equation}
\begin{aligned}\label{eq:trick}
	\bm \Phi_{n+1}(\by^n,\by^{(1)}(\by^n),\by^{n+1}(\by^n))=&(1-\alpha)\by^n+\alpha\by^{(1)}\\&+\Delta t\bA\diag(\beta_{20}\by^n+\beta_{21}\by^{(1)})(\diag(\by^{(1)}))^{-s}(\diag(\by^{n}))^{s-1}\by^{n+1}-\by^{n+1}.
\end{aligned}
\end{equation}
From this, it is easy to see that
\begin{equation*}
	\begin{aligned}
	\bD^*_{n+1}\bm \Phi_{n+1}=(1-\alpha\beta) \Delta t\bA-\bI
	\end{aligned}
\end{equation*}
which is a nonsingular matrix since $\sigma(\bA)\tm \C^-$, and \eqref{eq:alphabeta_conditions} implies $1-\alpha\beta\geq \frac{1}{2\beta}>0$.
Finally, we introduce the expressions for the Jacobians into the formula \eqref{eq:Formula_Dg(y*)} resulting in
\begin{equation}\label{eq:Dg(y*)SSPMRPK2}
	\begin{aligned}
		\bD\bg(\by^*)=(\bI-(1-\alpha\beta)\Delta t\bA)^{-1}\Bigl(&(1-\alpha)\bI+\Delta t\bA((s-1)(1-\alpha\beta)+\beta_{20})\\&+\left(\alpha \bI+\Delta t\bA(-s(1-\alpha \beta)+\beta_{21})\right)(\bI-\beta \Delta t\bA)^{-1}\Bigr).
	\end{aligned}
\end{equation}
Since $\bD\bg(\by^*)$ is a rational function of $\bA$ and the identity matrix $\bI$, any eigenvector of $\bA$ with the eigenvalue $\lambda$ is consequently an eigenvector of $\bD\bg(\by^*)$ with the eigenvalue $R(\Delta t \lambda)$ where
\begin{equation*}
	\begin{aligned}
		R(z)=\frac{1 - \alpha + z((s - 1)(1-\alpha\beta ) + \beta_{20}) + \frac{\alpha + z(-s(1-\alpha\beta) + \beta_{21})}{1-\beta z}}{1 - (1-\alpha\beta)z}.
	\end{aligned}
\end{equation*}
From 
\begin{equation*}
	\beta_{20} = 1 - \frac{1}{2\beta} - \alpha\beta,\quad  \beta_{21} = \frac{1}{2\beta}\qta s = \frac{\alpha\beta^2 - \alpha\beta + 1}{\beta(1-\alpha\beta)}
\end{equation*}
elementary computations lead to
\begin{equation*}
	R(z)=\frac{-2 + (2\alpha\beta^2 - 2\alpha\beta + 1)z^2 - 2\beta(\alpha - 1)z}{2(1 + (\alpha\beta - 1)z)(\beta z - 1)}.
\end{equation*}
In summary, we obtain the following proposition.
\begin{prop}\label{Prop:SSPMPRK2_Dg(y*)}
	Let $\bg\from\R^N_{>0}\to\R^N_{>0}$ be the map given by the application of the second order SSPMPRK to the differential equation \eqref{eq:Dahlquist_System} with $\bm 1\in \ker(\bA^T)$.
	Then any $\by^*\in \ker(\bA)\cap \R^N_{>0}$ is a fixed point of $\bg$ and $\bg\in \mathcal{C}^2(\R^N_{>0},\R^N_{>0})$, whereby the first derivatives of $\bg$ are Lipschitz continuous in an appropriate neighborhood of $\by^*$. Moreover, all linear invariants are conserved and an eigenvalue $\lambda$ of $\bA$ corresponds to the eigenvalue $R(\Delta t\lambda)$ of the Jacobian of $\bg$ where
	\begin{equation}\label{eq:StabfunSSPMPRK2}
		R(z)=\frac{-2 + (2\alpha\beta^2 - 2\alpha\beta + 1)z^2 - 2\beta(\alpha - 1)z}{2(1 + (\alpha\beta - 1)z)(\beta z - 1)}.
	\end{equation}
\end{prop}
By this proposition, the SSPMPRK2($\alpha,\beta$) scheme satisfies all preconditions in order to apply Theorem \ref{Thm_MPRK_stabil}. Thus, we have to analyze the stability function $R$.
\begin{prop}\label{Prop:Stab_SSPMPRK2}
Let $R$ be defined by \eqref{eq:StabfunSSPMPRK2} with $\alpha,\beta$ satisfying \eqref{eq:alphabeta_conditions}.
	\begin{enumerate}
\item\label{item:a_propSSP2}  For any $\alpha>\frac{1}{2\beta}$, the set $\{z\in \C^-\mid \lvert R(z)\rvert \leq 1\}$ is bounded.
\item \label{item:b_propSSP2}	For all $\alpha<\frac{1}{2\beta}$ with $(\alpha,\beta)\neq(0,\frac12)$ we have $\lvert R(z)\rvert <1$ for all $z\in \C^-\setminus\{0\}$. 
\item\label{item:c_propSSP2} For $\alpha=\frac{1}{2\beta}$ or $(\alpha,\beta)=(0,\frac12)$ the relation $\lvert R(z)\rvert <1$ is true for all $z$ with $\re(z)<0$, and $\abs{R(z)}=1$ holds whenever $\re(z)=0$.
	\end{enumerate}

\end{prop}
\begin{proof}
For proving part \ref{item:a_propSSP2} of the proposition, we consider \eqref{eq:StabfunSSPMPRK2} which yields
\[\lim_{z\to-\infty} R(z)=\frac{2\alpha\beta^2 - 2\alpha\beta + 1}{2\beta(\alpha\beta-1)}=\frac{2\alpha\beta^2 - 2\alpha\beta + 1}{2\alpha\beta^2-2\beta}.\]
Note that for $\alpha=\frac{1}{2\beta}$, we obtain $\lim_{z\to-\infty} R(z)=\frac{\beta-1+1}{\beta-2\beta}=-1$. Finally, it is straightforward to verify
	\begin{equation*}
		\begin{aligned}
			\partial_\alpha\left(\lim_{z\to-\infty} R(z)\right)=	\partial_\alpha\left(\frac{2\alpha\beta^2 - 2\alpha\beta + 1}{2\beta(\alpha\beta-1)}\right)&=\frac{2\beta(\beta-1)2\beta(\alpha\beta-1)-(2\alpha\beta(\beta-1) + 1)2\beta^2}{4\beta^2(\alpha\beta-1)^2}\\&=\frac{1-2\beta}{(\alpha\beta-1)^2}<0,
		\end{aligned}
	\end{equation*}
	since $\beta\geq \frac12$. Therefore $\lim_{z\to-\infty} R(z)$ decreases with increasing $\alpha$. As a result, for any $\alpha>\frac{1}{2\beta}$, we find $\lim_{z\to-\infty} R(z)<-1$ and thus, there exists $z^*\in \C^-$ so that $\lvert R(z^*)\rvert>1$. Indeed, the set $\{z\in \C^-\mid \lvert R(z)\rvert \leq 1\}$ is bounded, as we find $\lvert R(z^*)\rvert>1$ for any $z^*\in \C^-$ with $\lvert z^*\rvert$ large enough.

We now focus on the derivation of the remaining statements, we
investigate $\lvert R(z)\rvert$ first on the imaginary axis.
	A technical but elementary computation for $z=\ii b$ and $b\in \R$ yields
	\[\lvert R(z)\rvert=\frac{1+b^4(\alpha\beta^2-\alpha\beta+\frac12)^2+b^2(1+(\alpha^2+1)\beta^2-2\alpha\beta)}{(1+(\alpha\beta-1)^2b^2)(\beta^2b^2+1)}.\]
Subtracting the denominator from the numerator leads to the expression
\begin{equation}\label{eq:expProveSSP2}
-(2\alpha\beta-2\beta-1)(2\alpha\beta-1)(2\beta-1)b^4.
\end{equation}
With respect to statement \ref{item:b_propSSP2}, we consider $\alpha<\frac{1}{2\beta}$ and $\beta>\frac12$, as $\beta=\frac12$ implies $\alpha=0$ due to equation \eqref{eq:alphabeta_conditions}. It follows that $2\beta-1>0$. Due to $\alpha<\frac{1}{2\beta}$, we see $2\alpha\beta <1$ and $2\alpha\beta-2\beta-1<1-2\beta-1<0$, so that the whole product \eqref{eq:expProveSSP2} becomes negative, whenever $z=\ii b\neq 0$. This is equivalent to $\lvert R(z)\rvert<1$ on the imaginary axis without the origin. Using the 
Phragm\'{e}n-Lindel\"of principle on the union of the origin and the interior of $\C^-$, we conclude that $\lvert R(z)\rvert\leq 1$ for all $z\in \C^-$, as $R$ is rational and no poles are located in $\C^-$. Furthermore, since $R$ is not constant, we conclude from the maximum modulus principle that there exists no $z_0$ in the interior of $\C^-$ with $\abs{R(z_0)}=1$, so that, $\abs{R(z)}<1$ holds for all $z\in \C^-\setminus\{0\}$.

The assertion \ref{item:c_propSSP2} can be proved in a similar way using \eqref{eq:expProveSSP2}. Indeed, in the case of $\alpha=\frac{1}{2\beta}$ or $(\alpha,\beta)=(0,\frac12)$, the product \eqref{eq:expProveSSP2} vanishes proving $\abs{R(z)}=1$ on the imaginary axis.
Once again taking advantage of the Phragm\'{e}n-Lindel\"of principle one can conclude $\abs{R(z)}<1$ in the interior of $\C^-$.
\end{proof}
As a result we obtain the following corollaries that are a direct consequence of the application of Theorem~\ref{Thm:_Asym_und_Instabil} and Theorem \ref{Thm_MPRK_stabil}.
\begin{cor}\label{Cor:SSPMPRK2stab}
	Let $\by^*$ be a positive steady state of the differential equation \eqref{eq:Dahlquist_System}. Then $\by^*$ is a fixed point of the SSPMPRK2($\alpha,\beta$) scheme and the following holds:
	\begin{enumerate}
		\item
For any $\alpha>\frac{1}{2\beta}$, the stability region of the SSPMPRK2($\alpha,\beta$) method is bounded.
\item  For all $\alpha<\frac{1}{2\beta}$ with $(\alpha,\beta)\neq(0,\frac12)$, the fixed point $\by^*$ is stable for all $\Delta t>0$.
\item For $\alpha=\frac{1}{2\beta}$ or $(\alpha,\beta)=(0,\frac12)$, the fixed point $\by^*$ is stable for all $\Delta t>0$, 
if all nonzero eigenvalues of $\bA$ have a negative real part. 
	\end{enumerate}
\end{cor}
\begin{cor}\label{Cor:SSPMPRK2stab1}
	Let the unique steady state $\by^*$ of the initial value problem \eqref{eq:Dahlquist_System}, \eqref{eq:IC} be positive. Then the iterates of the SSPMPRK2($\alpha,\beta$) scheme locally converge towards $\by^*$ for all $\Delta t>0$, if any of the following condition holds:
	\begin{enumerate}
		\item $\alpha <\frac{1}{2\beta}$ and $(\alpha,\beta)\neq (0,\frac12)$.
		\item  $\alpha=\frac{1}{2\beta}$ or $(\alpha,\beta)= (0,\frac12)$  and 
	additionally all nonzero eigenvalues of $\bA$ have a negative real part.
	\end{enumerate}
\end{cor}
In order to illustrate the consequences of Corollary \ref{Cor:SSPMPRK2stab}, consider Figure \ref{Fig:alphabetagraph}, where due to \eqref{eq:alphabeta_conditions} all permitted pairs of $(\alpha,\beta)$ with $\beta\leq 5$ lie between the $\beta$-axis and the black curve. The blue graph is determined by $\alpha=\frac{1}{2\beta}$, and thus, separates pairs of parameters associated with unconditionally stable fixed points from those with bounded stability domains. The red rectangular with vertices $(0.2,3)$, $(0.2,3.5)$, $(0.24,3)$ and $(0.24,3.5)$ is located in that critical region, so that we further analyze the corresponding choices of parameters with the help of Figure \ref{Fig:StabregionSSPMPRK2}, where we plot the corresponding stability regions. One can observe that the chosen pairs of parameters from Figure \ref{Fig:alphabetagraph} that are closer to the blue graph are associated with a larger stability domain. The smallest stability region among the examples from Figure \ref{Fig:StabregionSSPMPRK2} are associated with the $(\alpha,\beta)$ pair at the top right corner of the red rectangular from Figure \ref{Fig:alphabetagraph}.
\begin{figure}[!h]
	\centering
\includegraphics[width=0.5\textwidth]{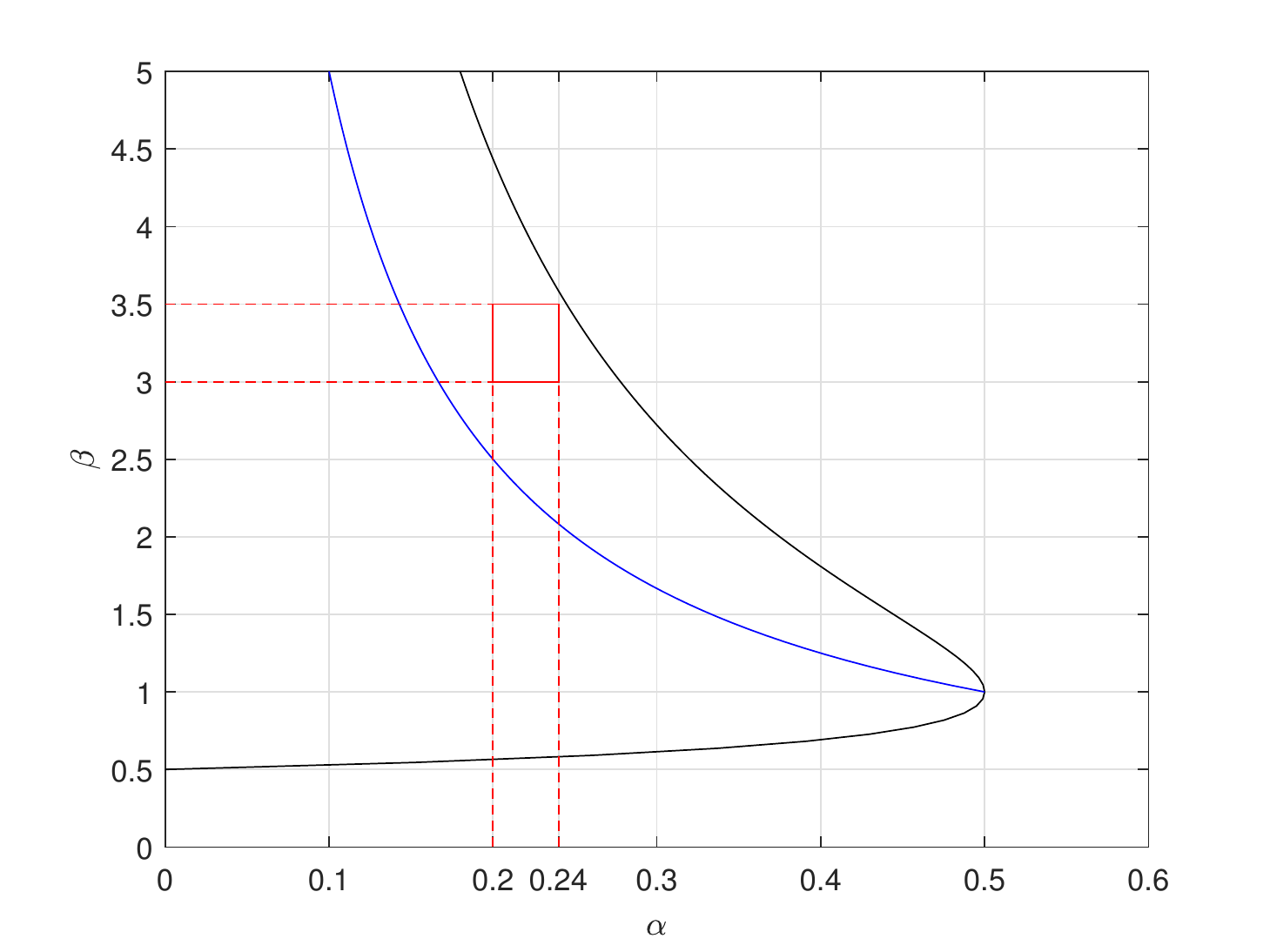}
\caption{The black curve is implicitly given by the function $\alpha(\beta)=\frac{1-\frac{1}{2\beta}}{\beta}$. The blue graph is determined by the equation $\alpha=\frac{1}{2\beta}$ and the red rectangular possesses the vertices $(\alpha,\beta)$ with $(0.2,3)$, $(0.2,3.5)$, $(0.24,3)$ and $(0.24,3.5)$ which lie between the black and blue curve.}\label{Fig:alphabetagraph}
\end{figure}
\begin{figure}[h!]
	\centering
\begin{subfigure}[c]{0.38\textwidth}
\includegraphics[width=\textwidth]{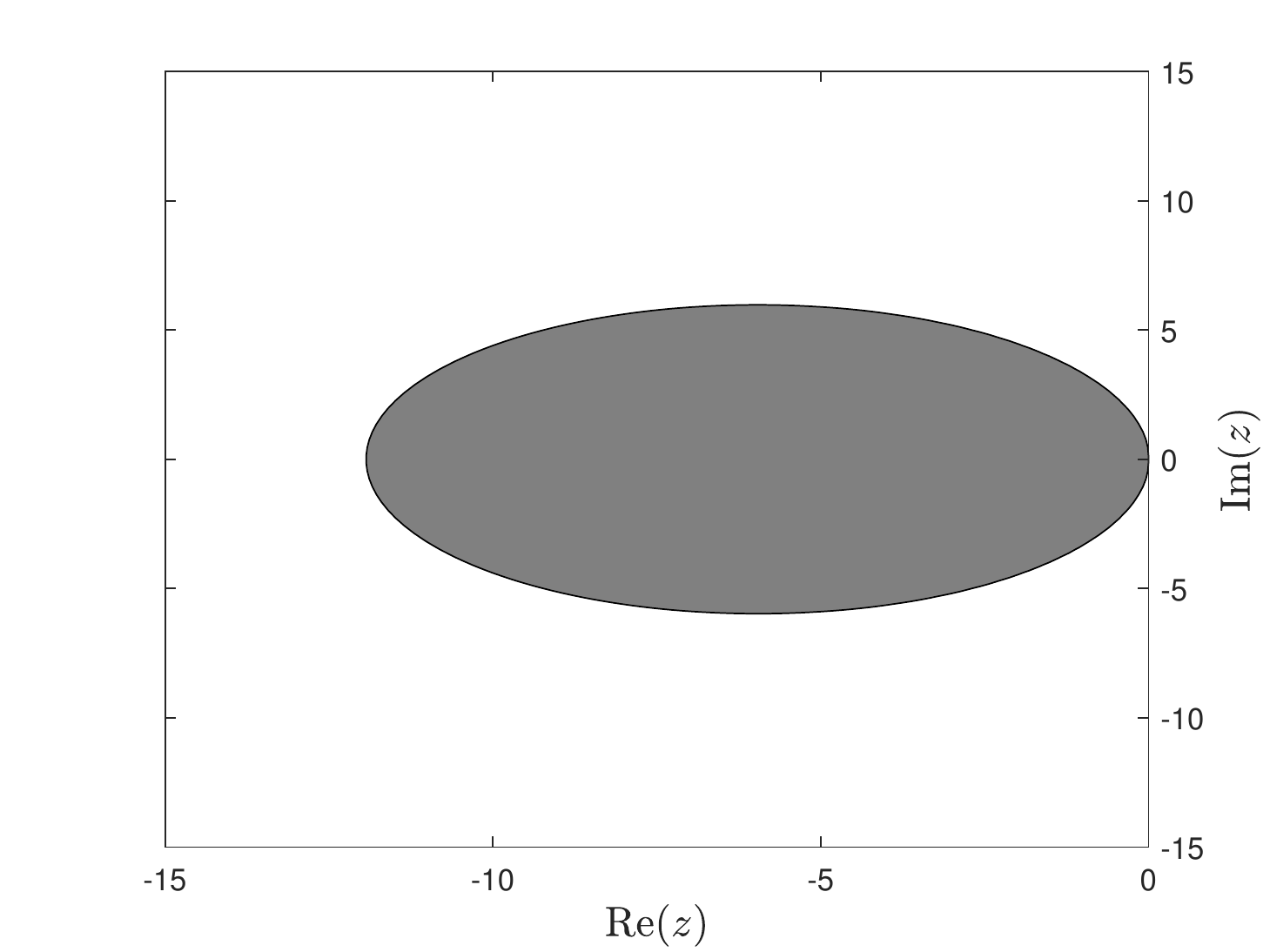}
\subcaption{$(\alpha,\beta)=(0.2,3)$}
\end{subfigure}
\begin{subfigure}[c]{0.38\textwidth}
\includegraphics[width=\textwidth]{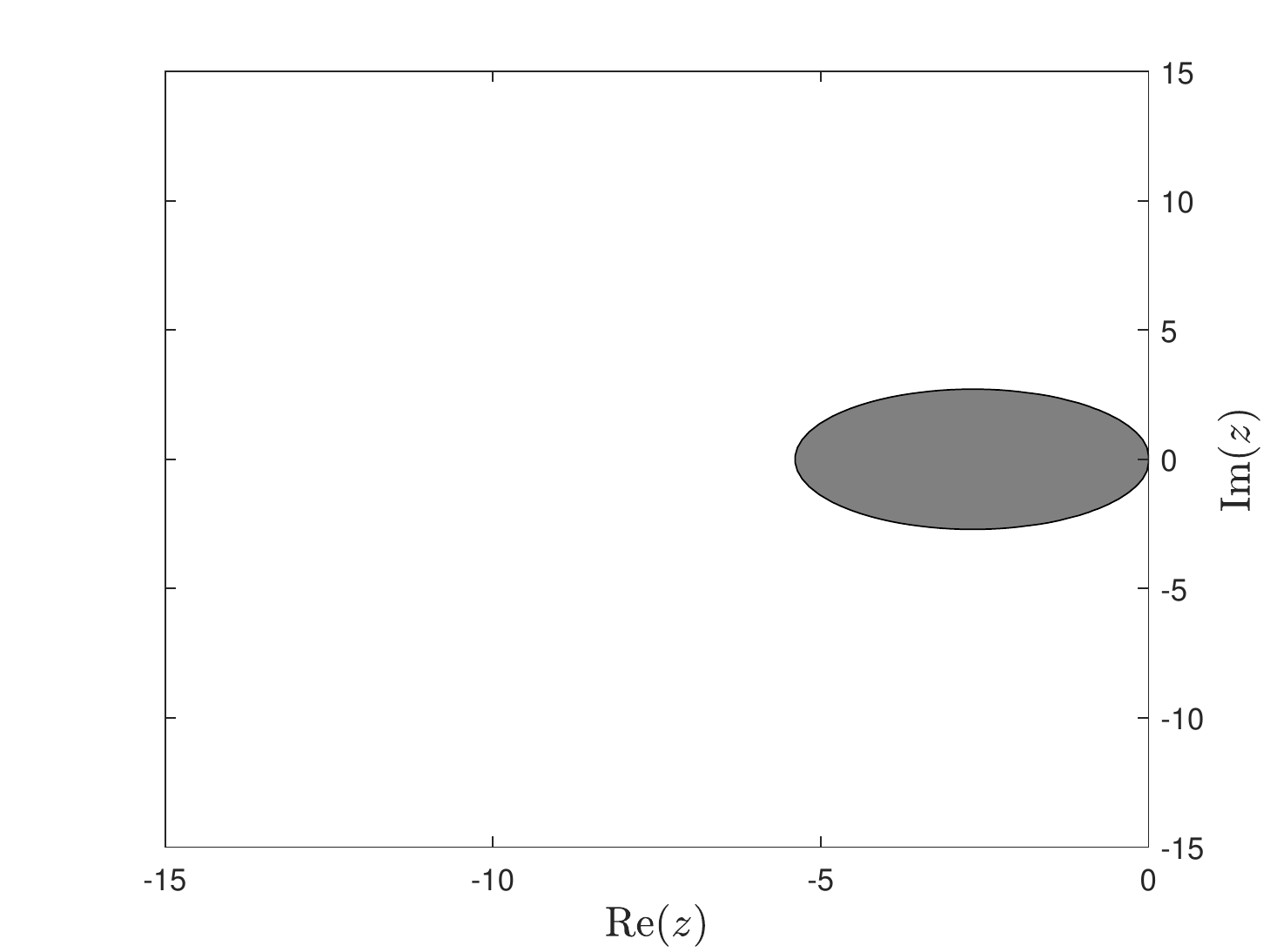}
\subcaption{$(\alpha,\beta)=(0.24,3)$}
\end{subfigure}\\
\begin{subfigure}[t]{0.38\textwidth}
\includegraphics[width=\textwidth]{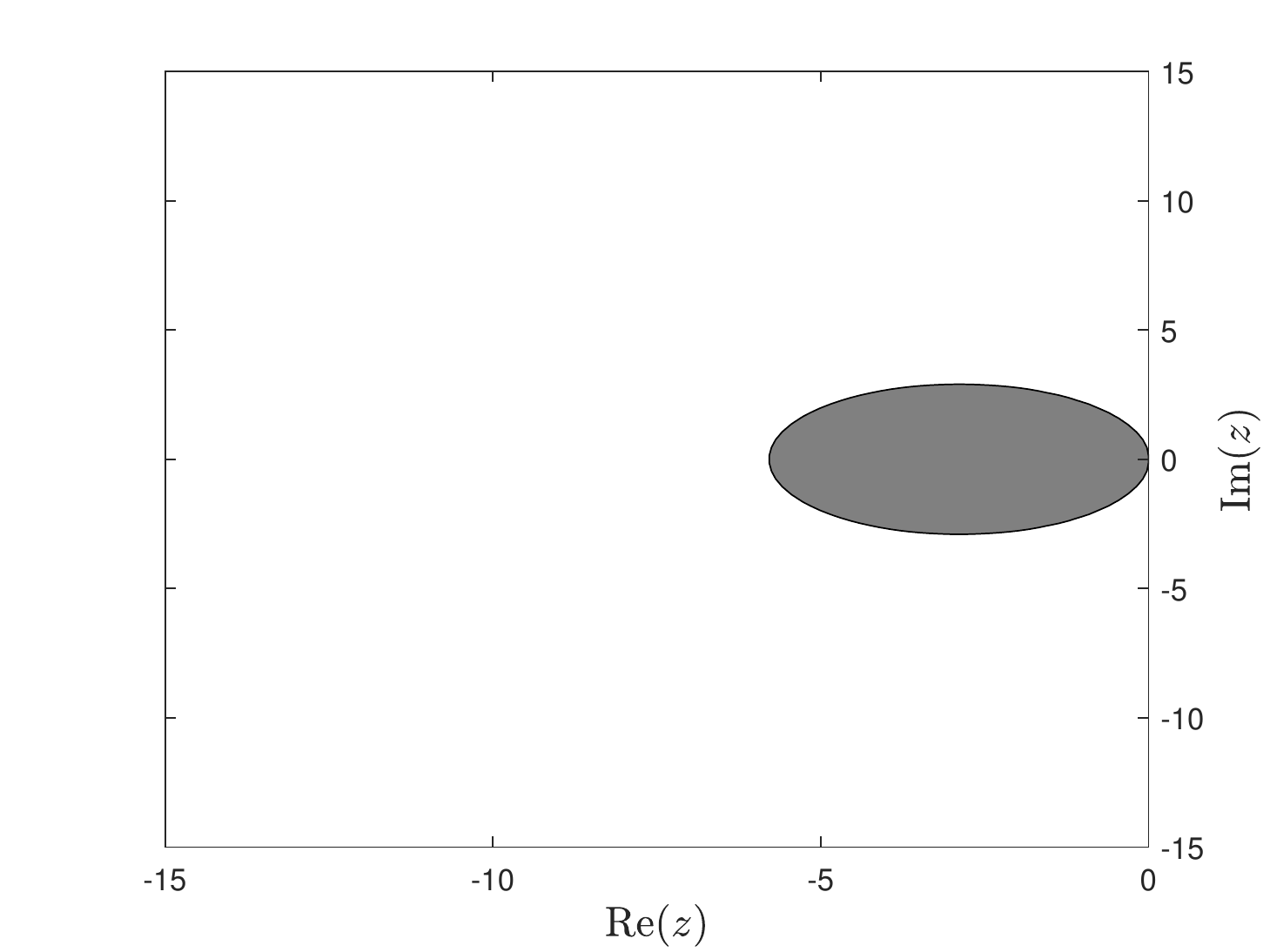}
\subcaption{$(\alpha,\beta)=(0.2,3.5)$}
\end{subfigure}
\begin{subfigure}[t]{0.38\textwidth}
\includegraphics[width=\textwidth]{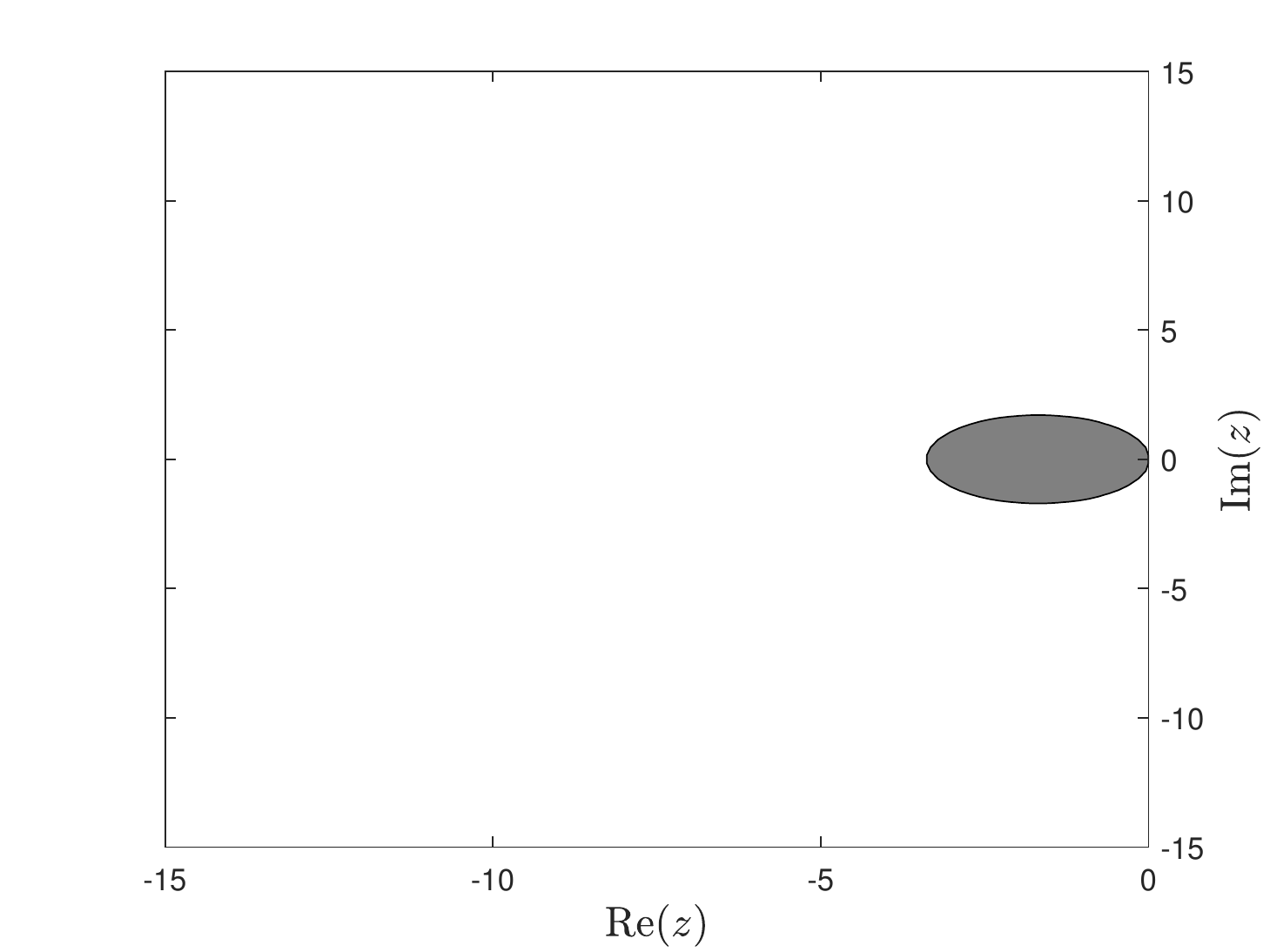}
\subcaption{$(\alpha,\beta)=(0.24,3.5)$}
\end{subfigure}
\caption{Different stability domains of the SSPMPRK2($\alpha,\beta$) method are plotted for $(\alpha,\beta)$ associated with the corners of the red rectangular from  Figure \ref{Fig:alphabetagraph}.}\label{Fig:StabregionSSPMPRK2}
\end{figure}

\subsection{Third order SSPMPRK Scheme}
In this section, we analyze the stability properties of the third order SSPMPRK scheme, introduced in \cite{MR3969000}, by using the same tools as for the second order scheme. We start by presenting the method, which can be written as
\begin{equation}\label{eq:SSPMPRK3}
\begin{aligned}
y_i^{(1)}=&\alpha_{10}y_i^n+\beta_{10} \Delta t\left(\sum_{j=1}^Np_{ij}(\by^n)\frac{y_j^{(1)}}{y_j^n}- \sum_{j=1}^Nd_{ij}(\by^n)\frac{y_i^{(1)}}{y_i^n}\right),\\
\rho_i=&n_1y_1^{(1)}+n_2y_i^n\left(\frac{y_i^{(1)}}{y_i^n}\right)^2,\\
y_i^{(2)}=&\alpha_{20}y_i^n+\alpha_{21} y_i^{(1)}+\Delta t\Biggl(\sum_{j=1}^N\left(\beta_{20}p_{ij}(\by^n)+\beta_{21}p_{ij}(\by^{(1)})\right)\frac{y_j^{(2)}}{\rho_j}- \sum_{j=1}^N\left(\beta_{20}d_{ij}(\by^n)+\beta_{21}d_{ij}(\by^{(1)})\right)\frac{y_i^{(2)}}{\rho_i}\Biggr),\\
a_i=&\eta_1y_i^n+\eta_2 y_i^{(1)}\\&+\Delta t\Biggl(\sum_{j=1}^N\left(\eta_3p_{ij}(\by^n)+\eta_{4}p_{ij}(\by^{(1)})\right)\frac{y_j^{(2)}}{(y_j^n)^{1-s}(y_j^{(1)})^s}-\sum_{j=1}^N\left(\eta_3d_{ij}(\by^n)+\eta_{4}d_{ij}(\by^{(1)})\right)\frac{y_i^{(2)}}{(y_i^n)^{1-s}(y_i^{(1)})^s}\Biggr),\\
\sigma_i=&a_i+\zeta y_i^n\frac{y_i^{(2)}}{\rho_i},\\
y_i^{n+1}=&\alpha_{30}y_i^n+\alpha_{31} y_i^{(1)}+\alpha_{32}y_i^{(2)}+\Delta t\Biggl(\sum_{j=1}^N\left(\beta_{30}p_{ij}(\by^n)+\beta_{31}p_{ij}(\by^{(1)})+\beta_{31}p_{ij}(\by^{(2)})\right)\frac{y_j^{n+1}}{\sigma_j}\\&- \sum_{j=1}^N\left(\beta_{30}d_{ij}(\by^n)+\beta_{31}d_{ij}(\by^{(1)})+\beta_{32}d_{ij}(\by^{(2)})\right)\frac{y_i^{n+1}}{\sigma_i}\Biggr),
\end{aligned}
\end{equation}
where we use the parameters
\begin{equation}\label{eq:SSPMPRK3Parameters}
\begin{aligned}
	\alpha_{10}&=1,&\alpha_{20} &= 9.2600312554031827\cdot10^{-1}, &\alpha_{21} &= 7.3996874459681783\cdot10^{-2},\\
\alpha_{30} &= 7.0439040373427619\cdot10^{-1},
&\alpha_{31} &= 2.0662904223744017\cdot10^{-10},
&\alpha_{32} &= 2.9560959605909481\cdot10^{-1},\\
\beta_{10} &= 4.7620819268131703\cdot10^{-1},
&\beta_{20} &= 7.7545442722396801\cdot10^{-2},
&\beta_{21} &= 5.9197500149679749\cdot10^{-1},\\
\beta_{30} &= 2.0044747790361456\cdot10^{-1},
&\beta_{31} &= 6.8214380786704851\cdot10^{-10},
&\beta_{32} &= 5.9121918658514827\cdot10^{-1},\\
\zeta&=0.62889380778287493358,
&\eta_1&=0.37110619221712506642 - \eta_2, &
\eta_3&=-1.2832127371313151768\eta_2 \\
&&&&&\hphantom{=}+ 0.6146025595987523739\\
\eta_4&=2.2248760403511226405, &n_1&=0.25690460257320105191, &
n_2&=1-n_1
\end{aligned}
\end{equation}
in accordance with \cite{MR3969000}.
Here, $\eta_2$ is a free parameter satisfying $\eta_2\in[0,r_1]$ with $r_1=0.37110619221712509$, so that we refer to this scheme as SSPMPRK3($\eta_2$).

As the first step, we apply this scheme to the linear test problem \eqref{eq:Dahlquist_System} and rewrite it in the
matrix-vector notation. For this, we again make use of equation \eqref{eq:-sum(dij)} and the fact that the production and destruction terms are linear, which results in
\begin{equation}\label{eq:SSPMPRK3MatrixVector}
	\begin{aligned}
		\bzero&=\bm \Phi_1(\by^n,\by^{(1)}) = \alpha_{10}\by^n+\beta_{10} \Delta t\bA\by^{(1)}-\by^{(1)},\\
		\bzero&=\bm \Phi_{\bm\rho}(\by^n,\by^{(1)},\bm \rho) = n_1\by^{(1)}+n_2(\diag(\by^{(1)}))^{2}(\diag(\by^{n}))^{-1}\bm 1-\bm \rho,\\
		\bzero&=\bm \Phi_2(\by^n,\by^{(1)},\bm\rho,\by^{(2)}) = \alpha_{20}\by^n+\alpha_{21}\by^{(1)}+\Delta t\bA\diag(\by^{(2)})(\diag(\bm \rho))^{-1}(\beta_{20}\by^n+\beta_{21}\by^{(1)})-\by^{(2)},\\
		\bzero&=\bm \Phi_{\bm a}(\by^n,\by^{(1)},\bm a) = \eta_1\by^n+\eta_2\by^{(1)}+\Delta t\bA\diag(\bm a)(\diag(\by^n))^{s-1}(\diag(\by^{(1)}))^{-s}(\eta_{3}\by^n+\eta_{4}\by^{(1)})-\bm a,\\
		\bzero&=\bm \Phi_{\bm \sigma}(\by^n,\by^{(2)},\bm \rho,\bm a,\bm \sigma) = \bm a+\zeta(\diag(\by^{(2)}))(\diag(\bm \rho))^{-1}\by^n-\bm \sigma,\\
		\bzero&=\bm \Phi_{n+1}(\by^n,\by^{(1)},\bm\rho,\by^{(2)},\by^{n+1})\\& = \alpha_{30}\by^n+\alpha_{31}\by^{(1)}+\alpha_{32}\by^{(2)}+\Delta t\bA\diag(\by^{n+1})(\diag(\bm \sigma)^{-1}(\beta_{30}\by^n+\beta_{31}\by^{(1)}+\beta_{32}\by^{(2)})-\by^{n+1},
	\end{aligned}
\end{equation}
where we omitted to write the arguments as functions of $\by^n$.
Nevertheless, we want to point out that all functions from above are $\mathcal C^2$-maps for positive arguments when the arguments are interpreted as independent variables. Thus, the map $\bg$, which is determined by solving linear systems, is in $\mathcal C^2$. Due to Remark \ref{rem:C2->C1}, the first derivatives are Lipschitz continuous for a sufficiently small neighborhood of $\by^*$. 

Also, we can prove $\bg(\bv)=\bv$ for all $\bv\in \ker(\bA)\cap \R^N_{>0}$ as follows. We know that $\bm \Phi_1(\by^*,\by^*)=\bzero$, and hence,  $\by^n=\by^*$ implies $\by^{(1)}=\by^*$  as $\by^{(1)}$ is uniquely determined by $\by^n$. Analogously, we conclude $\bm \rho(\by^*)=\by^*$ as $n_1+n_2=1$. As a consequence, we conclude from $a_{20}+a_{21}=1$ at machine precision that also $\by^{(2)}(\by^*)=\by^*$. However, $\bm a(\by^*)=(\eta_1+\eta_2)\by^*$, from which it follows that $\bm \sigma(\by^*)=(\eta_1+\eta_2)\by^*+z\by^*=\by^*$ since $\eta_1+\eta_2=1-\zeta$. Finally $\by^{n+1}(\by^*)=\bg(\by^*)=\by^*$ because of $\sum_{i=0}^2\alpha_{3i}=1$ is true at machine precision. 

In the following we use $a_{20}+a_{21}=1$, $\sum_{i=0}^2\alpha_{3i}=1$ and $\alpha_{10}=1$ as well as the values of the functions evaluated at $\by^*$ without further notice. 

Moreover, we can observe that $\bg$ conserves all linear invariants as follows. First, $\bn^T\bA=\bzero$ implies \[\bn^T\by^{(1)}=\alpha_{10}\bn^T \by^n+\beta_{10}\Delta t\bn^T\bA\by^{(1)}=\bn^T\by^n.\]
As a consequence, we obtain
\[\bn^T\by^{(2)}=\alpha_{20}\bn^T \by^n+\alpha_{21}\bn^T\by^{(1)}+\bzero=(\alpha_{20}+\alpha_{21})\bn^T\by^n=\bn^T\by^n.\]
Altogether, we find that $\bg$ is linear invariants preserving due to
\[\bn^T\bg(\by^n)=\bn^T\by^{n+1}=\sum_{i=0}^2\alpha_{3i}\bn^T\by^n+\bzero=\bn^T\by^n.\] Hence, also in the third order case, the map $\bg$ satisfies all conditions for applying Theorem \ref{Thm:_Asym_und_Instabil} and Theorem \ref{Thm_MPRK_stabil}. Therefore, we are now interested in computing the Jacobian of $\bg$, which can be done by using the same techniques and notations as for the second order SSPMPRK scheme.
Total differentiation of the last equation of \eqref{eq:SSPMPRK3MatrixVector} and solving for $\bD\bg(\by^*)$ yield
\begin{equation}\label{eq:Dg(y*)_Formula_SSP3}
\bD\bm\bg(\by^*)=-(\bD^*_{n+1}\bm \Phi_{n+1})^{-1}(\bD^*_n\bm \Phi_{n+1}+\bD^*_1\bm \Phi_{n+1}\bD^*\by^{(1)}+\bD^*_2\bm \Phi_{n+1}\bD^*\by^{(2)}+\bD^*_{\sigma}\bm \Phi_{n+1}\bD^*\bm \sigma),
\end{equation}
if $(\bD^*_{n+1}\bm \Phi_{n+1})^{-1}$ exists. Hence, we need formulas for $\bD^*\by^{(1)}, \bD^*\by^{(2)}$ and $\bD^*\bm \sigma$. We use the same strategies as for the second order scheme and obtain by means of a total differentiation of the corresponding equation in \eqref{eq:SSPMPRK3MatrixVector} the formulas
\begin{equation}\label{eq:Dy-Dsigma}
\begin{aligned}
\bD^*\by^{(1)}&=-(\bD_1^*\bm\Phi_1)^{-1}\bD_n^*\bm\Phi_1,\\
\bD^*\by^{(2)}&=-(\bD^*_2\bm \Phi_2)^{-1}(\bD^*_n\bm \Phi_2+\bD^*_1\bm \Phi_2\bD^*\by^{(1)}+\bD^*_{\rho}\bm \Phi_2\bD^*\bm\rho),\\
\bD^*\bm\sigma&=-(\bD^*_{\sigma}\bm \Phi_{\bm \sigma})^{-1}(\bD^*_n\bm \Phi_{\bm \sigma}+\bD^*_2\bm \Phi_{\bm \sigma}\bD^*\by^{(2)}+\bD^*_{\rho}\bm \Phi_{\bm \sigma}\bD^*\bm\rho+\bD^*_{a}\bm \Phi_{\bm \sigma}\bD^*\bm a),
\end{aligned}
\end{equation}
provided that the inverses exist. However, to compute the last two Jacobians, we now require to have knowledge about $\bD^*\bm \rho$ and $\bD^*\bm a$. These Jacobians can be obtained by
\begin{equation}\label{eq:Drho,Da}
\begin{aligned}
\bD^*\bm \rho&=-(\bD_{\bm\rho}^*\Phi_{\bm\rho})^{-1}(\bD^*_n\bm \Phi_{\bm\rho}+\bD^*_1\bm \Phi_{\bm\rho}\bD^*\by^{(1)}),\\
\bD^*\bm a&=-(\bD^*_{a}\bm \Phi_{\bm a})^{-1}(\bD^*_n\bm \Phi_{\bm a}+\bD^*_1\bm \Phi_{\bm a}\bD^*\by^{(1)}),
\end{aligned}
\end{equation}
if the expressions are defined. Starting off with the calculation of $\bD^*\by^{(1)}$, we obtain
\begin{equation*}
	\begin{aligned}
		\bD^*_n\bm \Phi_1=\alpha_{10}\bI, \quad \bD^*_1\bm \Phi_1=\beta_{10} \Delta t\bA-\bI.
	\end{aligned}
\end{equation*}
Since $\beta_{10}>0$ we can use \eqref{eq:Dy-Dsigma} to conclude that
\begin{equation*}
\bD^*\by^{(1)}=-(\beta_{10} \Delta t\bA-\bI)^{-1}\cdot \alpha_{10}\bI=(\bI-\beta_{10} \Delta t\bA)^{-1}
\end{equation*}
is defined.
Next we focus on $\bD^*\bm \rho$ so that we can compute $\bD^*\by^{(2)}$ afterwards. For this purpose, we use again that diagonal matrices commute and that $\diag(\bv)\bw=\diag(\bw)\bv$ holds. Hence, we find
\begin{equation*}
	\begin{aligned}
		\bD^*_n\bm \Phi_{\bm\rho}&=-n_2\bI, \quad \bD^*_1\bm \Phi_{\bm\rho}=(n_1+2n_2)\bI, \quad \bD_{\bm\rho}^*\Phi_{\bm\rho}=-\bI,
	\end{aligned}
\end{equation*}
and due to \eqref{eq:Drho,Da},
\begin{equation*}
\bD^*\bm \rho=-n_2\bI+(n_1+2n_2)(\bI-\beta_{10} \Delta t\bA)^{-1}.
\end{equation*}
The computation of the following Jacobians requires the same technique as described in equations \eqref{eq:ProductRuleDiag} and \eqref{eq:trick}, from which we get
\begin{equation*}
	\begin{aligned}
		\bD^*_n\bm \Phi_2&=\alpha_{20}\bI+\beta_{20}\Delta t\bA, \quad \bD^*_1\bm \Phi_2=\alpha_{21} \bI+\beta_{21}\Delta t\bA, \quad \bD^*_{\rho}\bm \Phi_2=-(\beta_{20}+\beta_{21})\Delta t\bA\qta\\ \bD^*_2\bm \Phi_2&=(\beta_{20}+\beta_{21}) \Delta t\bA-\bI,
	\end{aligned}
\end{equation*}
respectively.
Since $\beta_{20}+\beta_{21}>0$ the inverse of $\bD^*_2\bm \Phi_2$ exists, and thus, $\bD^*\by^{(2)}$ is formally given by \eqref{eq:Dy-Dsigma}.

Next, we need $\bD^*\bm a$ in order to find $\bD^*\bm \sigma$. Exploiting once again the ideas from \eqref{eq:ProductRuleDiag} and \eqref{eq:trick}, we obtain
\begin{equation*}
	\begin{aligned}
		\bD^*_n\bm \Phi_{\bm a}&=\eta_1\bI+(\eta_1+\eta_2)\Delta t\bA((s-1)(\eta_3+\eta_4)+\eta_3), \quad \bD^*_1\bm \Phi_{\bm a}=\eta_2 \bI+(\eta_1+\eta_2)\Delta t\bA(-s(\eta_3+\eta_4)+\eta_4),\\ \bD^*_{a}\bm \Phi_{\bm a}&=(\eta_3+\eta_4) \Delta t\bA-\bI,
	\end{aligned}
\end{equation*}
where $\bD^*_{a}\bm \Phi_{\bm a}$ is nonsingular since $\eta_3+\eta_4>0$. Hence, with \eqref{eq:Drho,Da} even the Jacobian $\bD^*\bm a$ can be determined.

Computing
\begin{equation*}
	\begin{aligned}
		\bD^*_n\bm \Phi_{\bm \sigma}&=\zeta\bI, \quad \bD^*_2\bm \Phi_{\bm \sigma}=\zeta \bI,\quad \bD^*_{\rho}\bm \Phi_{\bm \sigma}=-\zeta\bI, \quad \bD^*_{\bm a}\bm \Phi_{\bm \sigma}=\bI,\quad \bD^*_{\bm \sigma}\bm \Phi_{\bm \sigma}=-\bI,
	\end{aligned}
\end{equation*}
we are able to obtain $\bD^*\bm \sigma$ from \eqref{eq:Dy-Dsigma}. Finally, the remaining Jacobians are given by
\begin{equation*}
	\begin{aligned}
		\bD^*_n\bm \Phi_{n+1}&=\alpha_{30}\bI+\beta_{30}\Delta t\bA, \quad \bD^*_1\bm \Phi_{n+1}=\alpha_{31} \bI+\beta_{31}\Delta t\bA,\quad \bD^*_2\bm \Phi_{n+1}=\alpha_{32}\bI+\beta_{32}\Delta t\bA,\\ \bD^*_{\bm \sigma}\bm \Phi_{n+1}&=-\Delta t\bA\sum_{i=0}^2\beta_{3i}, \quad \bD^*_{n+1}\bm \Phi_{n+1}= \Delta t\bA\sum_{i=0}^2\beta_{3i}-\bI
	\end{aligned}
\end{equation*} with $\sum_{i=0}^2\beta_{3i}>0$, so that we are now in the position to compute $\bD\bg(\by^*)$ using \eqref{eq:Dg(y*)_Formula_SSP3}. As all the matrices occurring within the expressions of the Jacobians above are either the identity matrix $\bI$ or the system matrix $\bA$ from \eqref{eq:Dahlquist_System}, the stability function for the third order SSPMPRK scheme can easily be computed by calculating $\bD\bg(\by^*)$ and substituting $\Delta t\bA$ by $\Delta t\lambda=z$, so that we end up with the stability function $R(\Delta t\lambda)=R(z)$ that reads
\begin{equation*}
\begin{aligned}
R(z)=&\frac{1}{1-z\sum_{i=0}^2\beta_{3i}}\Biggl(\alpha_{30}+\beta_{30}z+\frac{\alpha_{31}+\beta_{31}z}{1-\beta_{10} z}+(\alpha_{32}+\beta_{32}z)P(z)-z\sum_{i=0}^2\beta_{3i}\Biggl(\zeta+\zeta P(z)-\zeta\left(-n2+\frac{n_1+2n_2}{1-\beta_{10}z}\right)\\
&+\frac{1}{1-(\eta_3+\eta_4)z}\left(\eta_1+(\eta_1+\eta_2)z\Bigl((s-1)(\eta_3+\eta_4)+\eta_3\Bigr)+\frac{\eta_2+(\eta_1+\eta_2)z\bigl(-s(\eta_3+\eta_4)+\eta_4\bigr)}{1-\beta_{10}z}\right)\Biggr) \Biggr),
\end{aligned}
\end{equation*}
where
\[ P(z)=\frac{1}{1-(\beta_{20}+\beta_{21})z}\Biggl(\alpha_{20}+\beta_{20}z+\frac{\alpha_{21}+\beta_{21}z}{1-\beta_{10}z}
-(\beta_{20}+\beta_{21})z\left(-n_2+\frac{n_1+2n_2}{1-\beta_{10}z}\right)\Biggr).\]
Before a detailed investigation of the stability function $R$, we summarize the above derived results by means of the following proposition.
\begin{prop}\label{Prop:SSPMPRK3_Dg(y*)}
	Let $\bg\from\R^N_{>0}\to\R^N_{>0}$ be given by the application of SSPMPRK3($\eta_2$) to the differential equation \eqref{eq:Dahlquist_System} with $\bm 1\in \ker(\bA^T)$.
	Then any $\by^*\in \ker(\bA)\cap \R^N_{>0}$ is a fixed point of $\bg$ and $\bg\in \mathcal{C}^2(\R^N_{>0},\R^N_{>0})$, whereby the first derivatives of $\bg$ are Lipschitz continuous in an appropriate neighborhood of $\by^*$. Moreover, all linear invariants are conserved and an eigenvalue $\lambda$ of $\bA$ corresponds to the eigenvalue $R(\Delta t\lambda)$ of the Jacobian of $\bg$ where
	\footnotesize{\begin{equation}\label{eq:StabfunSSPMPRK3}
		\begin{aligned}
	R(z)=&\frac{1}{1-z\sum_{i=0}^2\beta_{3i}}\Biggl(\alpha_{30}+\beta_{30}z+\frac{\alpha_{31}+\beta_{31}z}{1-\beta_{10} z}+(\alpha_{32}+\beta_{32}z)P(z)-z\sum_{i=0}^2\beta_{3i}\Biggl(\zeta+\zeta P(z)-\zeta\left(-n_2+\frac{n_1+2n_2}{1-\beta_{10}z}\right)\\
	&+\frac{1}{1-(\eta_3+\eta_4)z}\left(\eta_1+(\eta_1+\eta_2)z\Bigl((s-1)(\eta_3+\eta_4)+\eta_3\Bigr)+\frac{\eta_2+(\eta_1+\eta_2)z\bigl(-s(\eta_3+\eta_4)+\eta_4\bigr)}{1-\beta_{10}z}\right)\Biggr) \Biggr),\\
	P(z)=&\frac{1}{1-(\beta_{20}+\beta_{21})z}\Biggl(\alpha_{20}+\beta_{20}z+\frac{\alpha_{21}+\beta_{21}z}{1-\beta_{10}z}
	-(\beta_{20}+\beta_{21})z\left(-n_2+\frac{n_1+2n_2}{1-\beta_{10}z}\right)\Biggr),
\end{aligned}
	\end{equation}}
and the parameters are given in \eqref{eq:SSPMPRK3Parameters}.
\end{prop}
Next, we will prove that the third order SSPMPRK scheme possesses stable fixed points for all $\eta_2\in[0,r_1]$ when applied to the test equation.
\begin{prop}
The stability function $R(z)$ of the third order SSPMPRK scheme satisfies $R(0)=1$ and $\lvert R(z)\rvert <1$ for all $z\in \C^-\setminus\{0\}$ up to machine precision.
\end{prop}
\begin{proof}
It is straightforward to see that $R(0)=\alpha_{30}+\alpha_{31}+\alpha_{32}(\alpha_{20}+\alpha_{21})$ holds true. Up to machine precision, we obtain $\alpha_{20}+\alpha_{21}=1$ and $\alpha_{30}+\alpha_{31}+\alpha_{32}=1$, so that $R(0)=1$. Also, as $\alpha_{ij},\beta_{ij},\eta_3+\eta_4>0$, see \eqref{eq:SSPMPRK3Parameters}, no poles of $R$ are located in $\C^-$.
Furthermore, by a technical calculation we can rewrite $R$ to receive
\begin{equation*}
R(z)=\frac{\sum_{j=0}^4a_jz^j}{\sum_{j=0}^4b_jz^j},
\end{equation*}
where, for $\eta_2\in [0,r_1]\tm [0,\frac12)$ the coefficients are given by
\begin{equation*}
\begin{aligned}
	a_0&=\frac{0.47620819268131705757\eta_2 - 1.0537480911094115481}{0.47620819268131703\eta_2 - 1.0537480911094114871}>0,\\
	a_1&=\frac{-3.1507612671062001337\eta_2 + 3.9798736646158920698 + 0.61107641837494959323\eta_2^2}{0.47620819268131703\eta_2 - 1.0537480911094114871}<0,\\
	a_2&=\frac{2.4343280828365809236\eta_2 - 2.5818776483048969774 - 0.57282016379130601724\eta_2^2}{0.47620819268131703\eta_2 - 1.0537480911094114871}>0,\\
	a_3&=\frac{0.6548068883713070549\eta_2 - 0.81603432814746304744 - 0.1292603911580354457\eta_2^2}{0.47620819268131703\eta_2 - 1.0537480911094114871}>0,\\
a_4&=\frac{-0.59574557514538034065\eta_2 + 0.64052974630005292675 + 0.13841284380675759373\eta_2^2}{0.47620819268131703\eta_2 - 1.0537480911094114871}<0,\\
b_0&=1>0,\\
b_1&=-4.7768739020212929733 + 1.2832127371313151768\eta_2<0,\\
b_2&=6.7270587897458664634 - 2.4860903284764154151\eta_2>0,\\
b_3&=-3.7332290665687486456 + 1.5730472371819288192\eta_2<0,\\
b_4&= 0.71670702950202557447-0.32389312216150656421\eta_2>0.
\end{aligned}
\end{equation*}
Substituting $z=\ii y$ with $y\in \R\setminus\{0\}$, we find
\begin{equation*}
\lvert R(\ii y)\rvert^2=\frac{(y^4a_4-y^2a_2+a_0)^2+(-y^3a_3+ya_1)^2}{(y^4b_4-y^2b_2+1)^2+(-y^3b_3+yb_1)^2},
\end{equation*}
so that $\lvert R(\ii y)\rvert<1$ is equivalent to
\begin{equation*}
(y^4a_4-y^2a_2+a_0)^2+(-y^3a_3+ya_1)^2-((y^4b_4-y^2b_2+1)^2+(-y^3b_3+yb_1)^2)<0.
\end{equation*}
Collecting powers of $y$, one can rewrite the above inequality in the form
\begin{equation*}
(a_4^2-b_4^2)y^8+(-2a_2a_4+a_3^2+2b_2b_4-b_3^2)y^6+(2a_0a_4-2a_1a_3+a_2^2+2b_1b_3-b_2^2-2b_4)y^4+(-2a_0a_2+a_1^2-b_1^2+2b_2)y^2+a_0^2-1<0.
\end{equation*}
At machine precision, we obtain $a_0=1$, so that $a_0^2-1=0$. Next, our strategy is to prove that all nonzero coefficients of $y^k$, in the following denoted by $c_k$, for $k=2,4,6,8$ are negative.

For $\eta_2\leq r_1< \frac12$, it suffices for our argument to round to two decimal places in the following expressions, which yields
\begin{equation*}
\begin{aligned}
	c_8&\approx\frac{-0.74\eta_2^2+1.18\eta_2+0.20\eta_2^3-0.71-0.02\eta_2^4}{(\eta_2-2.21)^2},\\
	c_6&\approx\frac{-3.45\eta_2^2 + 5.70\eta_2+0.92\eta_2^3 - 3.51 - 0.09\eta_2^4}{(\eta_2 - 2.21)^2},\\
	c_4&\approx\frac{4.41(0.53\eta_2 - 0.43 + 0.03\eta_2^3 - 0.22\eta_2^2)}{(\eta_2 - 2.21)^2},\\
	10^{14}c_2&\approx\frac{\eta_2(\eta_2 - 1 - 0.2\eta_2^2 + 0.03\eta_2^3)}{(0.48\eta_2 - 1.05)^2}.
\end{aligned}
\end{equation*}
First of all, the denominators occurring in any of the above $c_k$ are positive. Also, positive terms in the numerator are multiplied with powers of $\eta_2<\frac12$ and thus are smaller than the absolute value of the corresponding constant, which is always negative. This holds true even if the rounding error is taken into account, i.\,e.\ after adding $10^{-2}$ to positive terms and subtracting it from negative expressions. This proves that $c_k<0$, and thus, $\abs{R(\ii y)}<1$ for all $y\in \R\setminus\{0\}$.

Finally, we can perform the same steps as in the proof of Proposition~\ref{Prop:Stab_SSPMPRK2} to conclude even $\lvert R(z)\rvert<1$ for all $z\in \C^-\setminus\{0\}$ by applying the Phragm\'{e}n-Lindel\"of principle for the union of the origin and the interior of $\C^-$, as well as the maximum modulus principle.
\end{proof}
As an immediate consequence of this proposition in combination with Theorem \ref{Thm:_Asym_und_Instabil} and Theorem \ref{Thm_MPRK_stabil}, we obtain the following results.
\begin{cor}\label{Cor:SSPMPRK3stab}
	Let $\by^*$ be a positive steady state of the differential equation \eqref{eq:Dahlquist_System}. Then $\by^*$ is a stable fixed point of the SSPMPRK3($\eta_2$) scheme for all $\Delta t>0$ and $\eta_2\in [0,r_1]$.
\end{cor}
\begin{cor}\label{Cor:SSPMPRK3stab1}
	Let the unique steady state $\by^*$ of the initial value problem \eqref{eq:Dahlquist_System}, \eqref{eq:IC} be positive. Then the iterates of SSPMPRK3($\eta_2$) locally converge towards $\by^*$ for all $\Delta t>0$ and $\eta_2\in [0,r_1]$.
\end{cor}

	\section{Numerical Experiments}\label{sec:Num_Tests}
	In order to verify the stability properties of the second and third order SSPMPRK schemes as stated in the Corollaries \ref{Cor:SSPMPRK2stab}, \ref{Cor:SSPMPRK2stab1}, \ref{Cor:SSPMPRK3stab} and \ref{Cor:SSPMPRK3stab1}, we consider three linear positive and conservative PDS test cases introduced in \cite{IKM22Sys}.

All systems matrices have an eigenvalue $\lambda=0$, since the test problems are conservative. Furthermore, the following test cases are chosen in such a way that all nonzero eigenvalues either lie in $\R^-$ or in $\C^-\setminus\R^-$. Moreover, as we proved that the SSPMPRK schemes conserve all linear invariants when applied to a linear system, we also consider a test problem with two linear invariants.\\

\subsection*{Test problem with exclusively real eigenvalues}
The linear initial value problem 
\begin{equation}
	\by'=100\Vec{-2& 1 &1\\1 &-4 &1\\1 &3 &-2}\by,\quad  \by(0)=\Vec{1\\ 9\\5}.\label{eq:initProbReal}
\end{equation}
 contains a system matrix, which has only positive off-diagonal elements and is therefore a so-called Metzler matrix. Due to the positive initial values, this ensures that each component of the solution of the initial value problem is positive for all times. By a straightforward calculation of the eigenvalues $\lambda_1=0$, $\lambda_2=-300$ and $\lambda_3=-500$ of the system matrix as well as their associated eigenvectors, the solution reads \begin{equation}
 	\by(t)=c_1\begin{pmatrix}5\\ 3\\ 7\end{pmatrix}+c_2e^{-300t}\begin{pmatrix}-1\\ 0\\ 1\end{pmatrix}+c_3e^{-500t}\begin{pmatrix}0\\ -1\\ 1\end{pmatrix}\label{eq:exsolReal}
 \end{equation}
 with coefficients $c_1=1$, $c_2=4$ and $c_3=-6$ determined by the initial condition. Since only non-positive eigenvalues are present and the absolute values of the negative eigenvalues are large, there is a fast convergence to the equilibrium state 
 \begin{equation*}
 	\by^*=\lim_{t\to\infty}\by(t)=\Vec{5\\3\\7}
 \end{equation*}
 as depicted in Figure \ref{Fig:initProbReal}. Furthermore the zero eigenvalue is simple, and hence there exists exactly one linear invariant, which is given by $\bm 1^T\by$ due to the fact that the sum of the elements in each column of the system matrix is always vanishing. This so-called conservativity  can also be observed in Figure \ref{Fig:initProbReal}.

\begin{figure}[h!]
	\centering
	\includegraphics[width=0.5\textwidth]{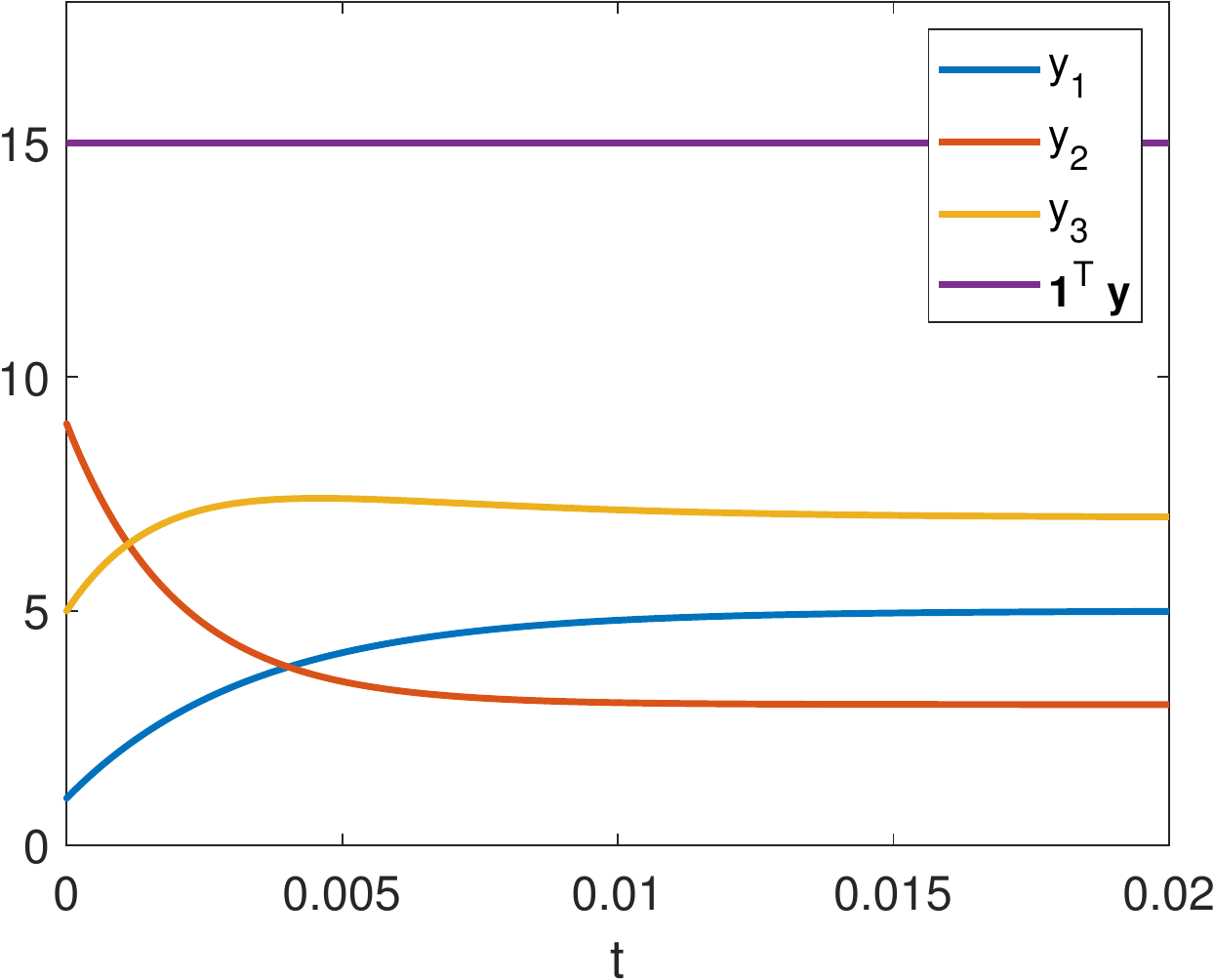}
	\caption{Exact solution \eqref{eq:exsolReal} of the initial value problem \eqref{eq:initProbReal} and the linear invariant $\bm 1^T\by$.}\label{Fig:initProbReal}
\end{figure}

\subsection*{Test problem with complex eigenvalues}

As the second test case, we consider the conservative system 
\begin{equation}
	\by'=100\Vec{-4& 3 &1\\2 &-4 &3\\2 &1 &-4}\by,\quad  \by(0)=\Vec{9\\ 20\\8}\label{eq:initProbIm}.
\end{equation}
Again, the system matrix is a Metzler matrix, so that the solution of the initial value problem is always positive due to the positive initial conditions. Considering the eigenvalues $\lambda_1=0$ , $\lambda_2=100(-6+\ii)$ and $\lambda_3=\overline{\lambda_2}$ as well as the corresponding eigenvectors of the system matrix, the solution can be written in the form
\begin{equation}
	\begin{aligned}
		\by(t)=&\begin{pmatrix}13\\ 14\\ 10\end{pmatrix}-2e^{-600t}\left(\cos \left(100t\right)\begin{pmatrix}-1\\ 0\\ 1\end{pmatrix}-\sin \left(100t\right)\begin{pmatrix}1\\ -1\\ 0\end{pmatrix}\right)\\&-6e^{-600t}\left(\cos \left(100t\right)\begin{pmatrix}1\\ -1\\ 0\end{pmatrix}+\sin \left(100t\right)\begin{pmatrix}-1\\ 0\\ 1\end{pmatrix}\right).\label{eq:exsolIm}
	\end{aligned}
\end{equation}
The nonzero complex eigenvalues have a negative real part with a large absolute value. Hence, one can expect a rapid convergence of the solution to the steady state given by
\begin{equation*}
	\by^*=\lim_{t\to\infty}\by(t)=\Vec{13\\14\\10}.
\end{equation*} 
Analogous to the first test case, the only linear invariant is $\bm 1^T\by$, which is presented together with the exact solution in Figure \ref{Fig:initProbIm}.

\begin{figure}[h!]
	\centering
	\includegraphics[width=0.5\textwidth]{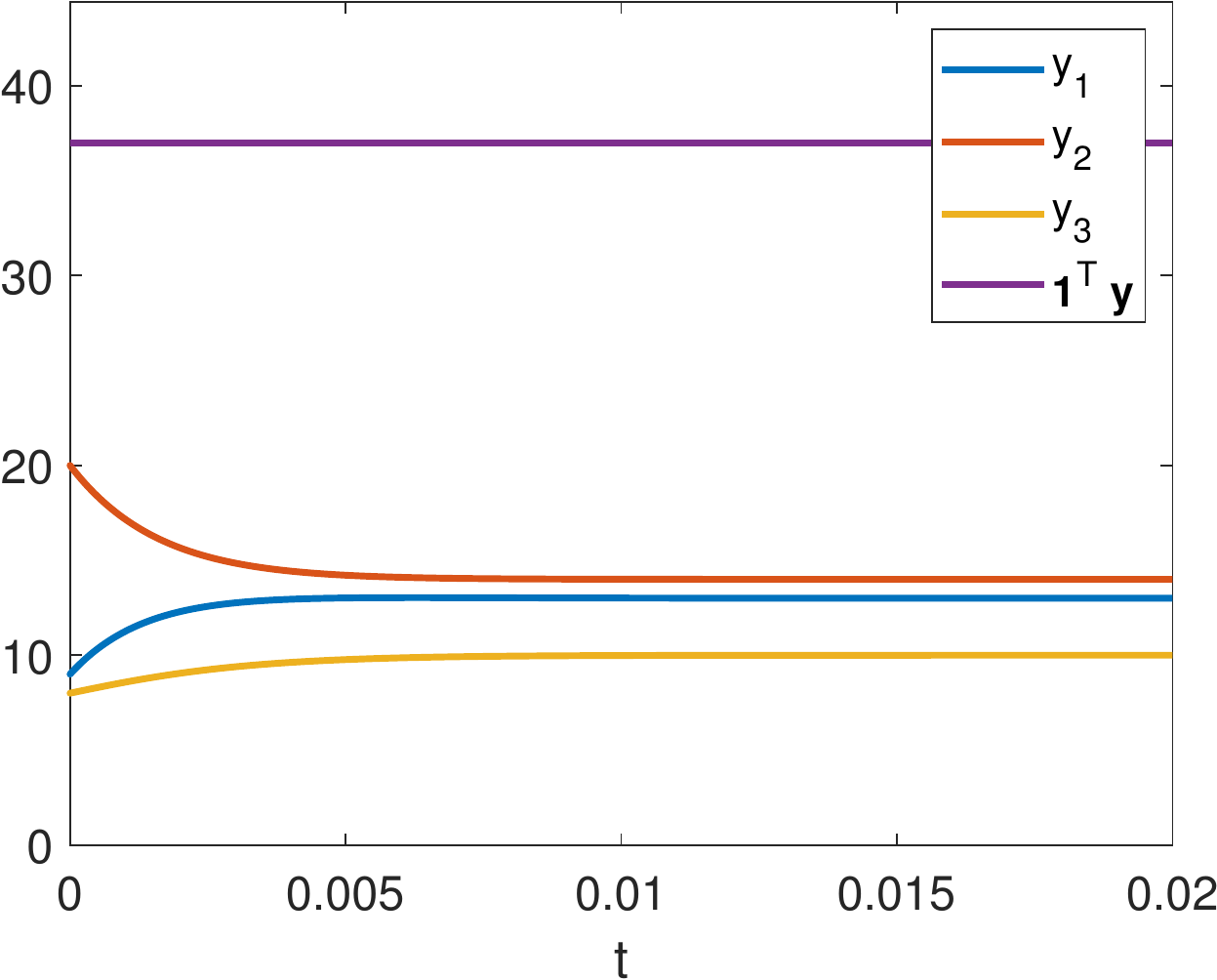}
	\caption{The exact solution \eqref{eq:exsolIm} of the initial value problem \eqref{eq:initProbIm} and the linear invariant $\bm 1^T\by$.}\label{Fig:initProbIm}
\end{figure}

\subsection*{Test problem with double zero eigenvalue}
Considering the linear initial value problem  
\begin{equation}
	\by'=100\Vec{-2& 0 &0 &1\\0 &-4 &3& 0\\0 &4& -3 &0\\ 2 & 0&0&-1}\by,\quad  \by(0)=\Vec{4\\ 1\\9\\1},\label{eq:initProb4dim}
\end{equation}
we are faced with a Metzler matrix including a double zero eigenvalue $\lambda_1 = \lambda_2=0$. Therefore, besides $\bm 1^T\by$, a second linear invariant $\bn^T\by$ with $\bn=(1,2,2,1)^T$ is present. Due to the remaining eigenvalues $\lambda_3=-300$ and $\lambda_3=-700$ and the associated eigenvectors of all eigenvalues, the solution of the initial value problem writes
\begin{equation}
	\by(t)=c_1\Vec{0\\1\\ \frac43\\0}+c_2\Vec{1\\0\\0\\2}+c_3e^{-700t}\Vec{0\\1\\-1\\0}+c_4e^{-300t}\Vec{1\\0\\ 0\\-1}\label{eq:exsol4dim}
\end{equation}
with coefficients
\begin{align*}
	c_1=\frac{30}{7},\quad c_2=\frac53, \quad c_3=-\frac{23}{7}\qta c_4=\frac73.
\end{align*}
Once again, a fast convergence to the equilibrium state
\begin{equation*}
	\by^*=\lim_{t\to\infty}\by(t)=c_1\Vec{0\\1\\ \frac43\\0}+c_2\Vec{1\\0\\0\\2}=\frac{1}{21}\Vec{7\\ 90\\ 120\\ 70}
\end{equation*}
takes place. The course of the solution together with the two linear invariants are shown in Figure \ref{Fig:initProb4dim}.
\begin{figure}[h!]
	\centering
	\includegraphics[width=0.5\textwidth]{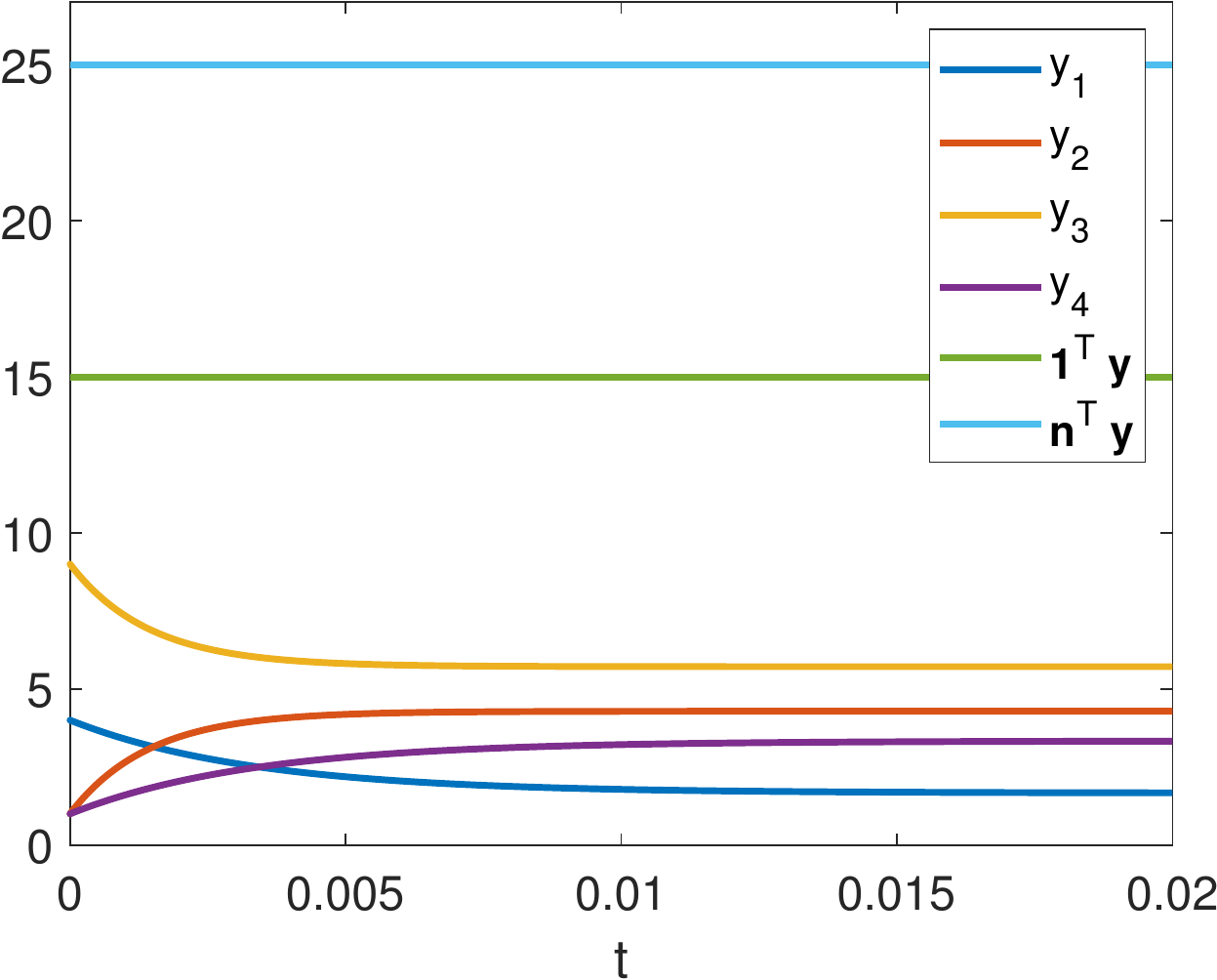}
	\caption{The exact solution \eqref{eq:exsol4dim} of the initial value problem \eqref{eq:initProb4dim} and the associated two linear invariants $\bm 1^T\by$ and $\bn^T\by$.}\label{Fig:initProb4dim}
\end{figure}

\begin{figure}
	\centering
	\includegraphics[scale=0.8]{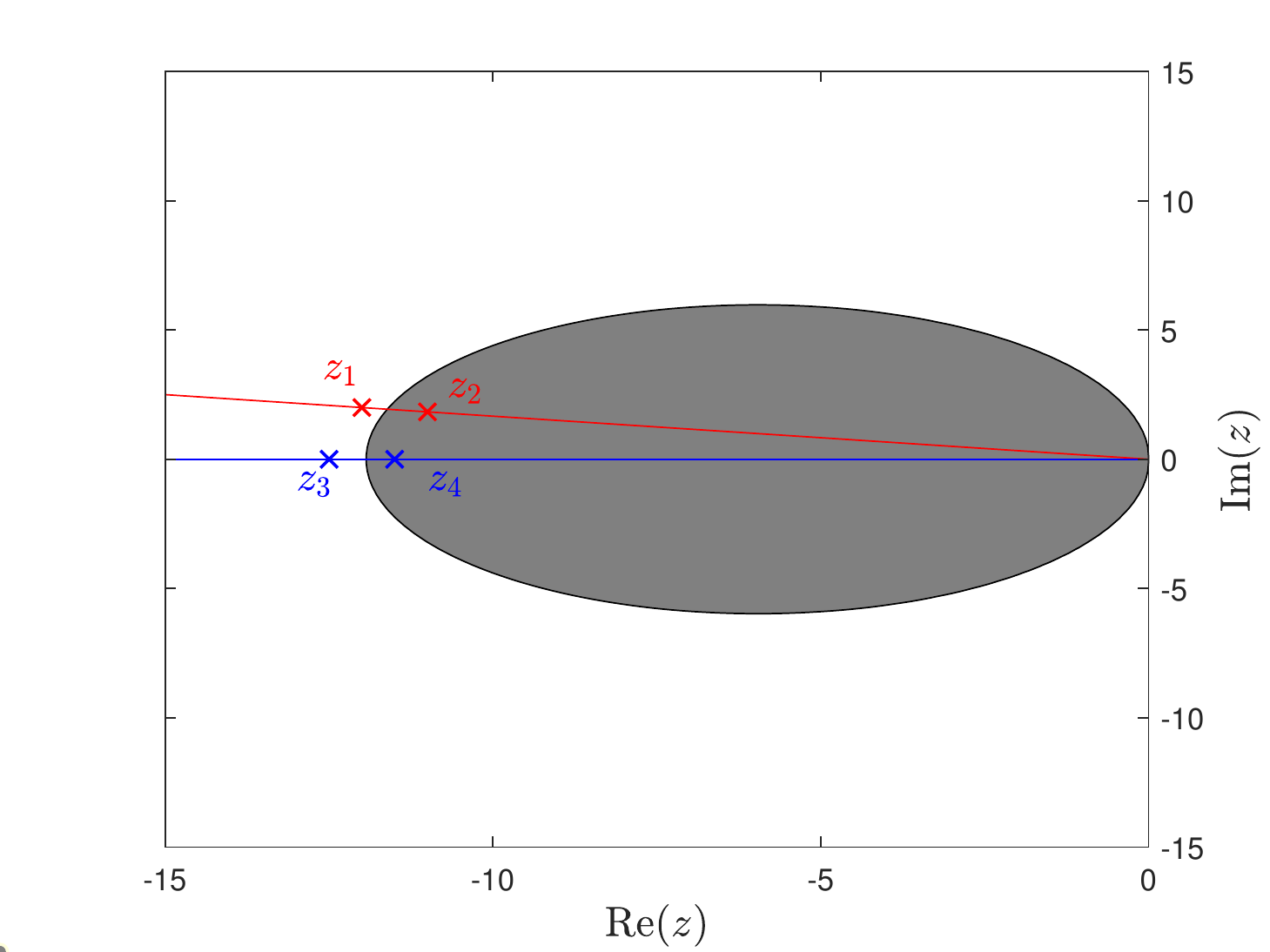}
	\caption{The stability region for the second order SSPMPRK scheme is depicted for the values $(\alpha,\beta)=(0.2,3)$. The red line is the set $\{x+\ii y\in\C^-\mid x+6y=0, \, x,y\in\R \}$ containing numbers of the form $a(-6+\ii)$ with $a\leq 0$. In particular, the red marked complex numbers are $z_1=2(-6+\ii)$ and $z_2=\frac{11}{6}(-6+\ii)$.
		The blue line is the set $\R^-$. In particular, the blue marked numbers are $z_3=-12.5$ and $z_4=-11.5$.}\label{Fig:Stabmarker}
\end{figure}

At this point we want to note that the presented test cases represent stiff problems due to the occurence of large absolute values of the corresponding eigenvalues. Hence, it is not surprising that the exact solution satisfies the inequality $\norm{\by(t)-\by^*}<2\cdot 10^{-2}$ at time $t=0.02$. 

Hereafter, we confirm numerically that SSPMPRK schemes are stable as claimed in Corollary \ref{Cor:SSPMPRK2stab} and Corollary \ref{Cor:SSPMPRK3stab}. 
Furthermore, we investigate the local convergence to the steady state solution as stated in Corollary \ref{Cor:SSPMPRK2stab1} and Corollary \ref{Cor:SSPMPRK3stab1} by choosing a comparably large time step size of $\Delta t=5$ for all examples, if not stated otherwise. 
In particular, we are interested in the properties of SSPMPRK3($\frac{1}{3}$) which is the preferred scheme presented in \cite{MR3969000}.
Moreover, we investigate SSPMPRK2($\alpha,\beta$) for three different pairs $(\alpha,\beta)$ covering all cases depicted in Figure \ref{Fig:alphabetagraph}.
For the case $\alpha>\frac{1}{2\beta}$ we choose the lower left vertex of the red rectangular from Figure \ref{Fig:alphabetagraph}, i.\,e.\ $(\alpha,\beta)=(0.2,3)$. In this case, we choose different time steps to demonstrate that the computed stability regions are correct. At this point we want to note that the eigenvalues of the system matrices from the test problems lie on the red or blue line depicted in Figure \ref{Fig:Stabmarker}. We scale the time step size $\Delta t$ in such a way that $\Delta t\rho(\bD\bg(\by^*))=z_i$ for some $i\in \{1,2,3,4\}$ so that for all test cases we consider the cases of stable as well as unstable fixed points.

 As a representative for the case $\alpha=\frac{1}{2\beta}$ we use $(\alpha,\beta)=(\frac12,1)$ which is the preferred choice presented in \cite{MR3934688}. Finally, we choose $(\alpha,\beta)=(0.1,1)$ satisfying $\alpha<\frac{1}{2\beta}$.

In the following figures, we plot the numerical approximations over different time intervals resulting in different numbers of total iterations $N_T$, whereby $\norm{\by^*-\by^{N_T}}<2\cdot 10^{-2}$ is satisfied. With a time step size of $\Delta t=5$, the third order SSPMPRK schemes satisfy this relation for $18\leq N_T\leq 25$ iterations with respect to all three test cases. Analyzing the second order SSPMPRK schemes with parameters $(\alpha,\beta)=(0.1,1)$ we find $N_T\approx 10$, and in the case of $(\alpha,\beta)=(\frac12,1)$ we have $N_T\approx 5000$. The investigation of the chose $(\alpha,\beta)=(0.2,3)$ is more delicate, as this pair is associated with a bounded stability domain. Nevertheless, choosing $\Delta t$ corresponding to the values $z_2$ and $z_4$ from figure \ref{Fig:Stabmarker}, we can observe that $N_T\approx 300$ holds.

In Figure \ref{Fig:SSPMPRK3initProblems}, the SSPMPRK3($\eta_2$) scheme is used to integrate the three test problems. The numerical experiments support the theoretical claims, i.\,e.\ the fixed points seem to be stable and locally attracting. Moreover, all linear invariants are conserved by the method.

In the subsequent figures, SSPMPRK2($\alpha, \beta$) schemes are used to solve the test problems. In all three figures \ref{Fig:MPRKinitProbReal}, \ref{Fig:MPRKinitProbIm}
 and \ref{Fig:MPRKinitProb4dim}, we can observe the same qualitative behavior. In the upper left plot, the value of $N_T$ is by far the biggest so that the preferred choice of $(\alpha,\beta)=(\frac{1}{2},1)$ seems to be the least damping scheme. Changing the value of $\alpha$ to $0.1$ results in a faster convergence towards the steady state solution, even for a time step size of $\Delta t=5$. 
 
 In the lower two figures in figures \ref{Fig:MPRKinitProbReal}, \ref{Fig:MPRKinitProbIm}
 and \ref{Fig:MPRKinitProb4dim}, the pair $(\alpha,\beta)$ lies in the critical region where the stability domain is bounded. If $\Delta t$ is chosen in such a way that $\Delta t\rho(\bD\bg(\by^*))=z_i$ for $i=2$ or $i=4$, respectively, see Figure \ref{Fig:Stabmarker}, the numerical approximations behave as expected converging towards the corresponding steady state which is a stable fixed point of the method. However, increasing $\Delta t$ by approximately $2\cdot 10^{-3}$, we find that $\Delta t\rho(\bD\bg(\by^*))=z_i$ for $i=1$ or $i=3$, respectively. As a result, even when we modify the starting vector to be $\by^0=\by^*+10^{-5}\bv$ with $\bv=(1,-2,1)^T$ in Figures \ref{Fig:MPRKinitProbReal} and \ref{Fig:MPRKinitProbIm}, or $\bv=(1,-1,1,-1)^T$ in Figure \ref{Fig:MPRKinitProb4dim}, the numerical approximation diverges from the steady state as predicted by the presented theory.
 
 Altogether, the numerical experiments support very well the theoretical results from Section \ref{sec:stab_SSPMPRK}.

\begin{figure}[!h]
	\centering
	\begin{subfigure}[t]{0.49\textwidth}
	\includegraphics[width=\textwidth]{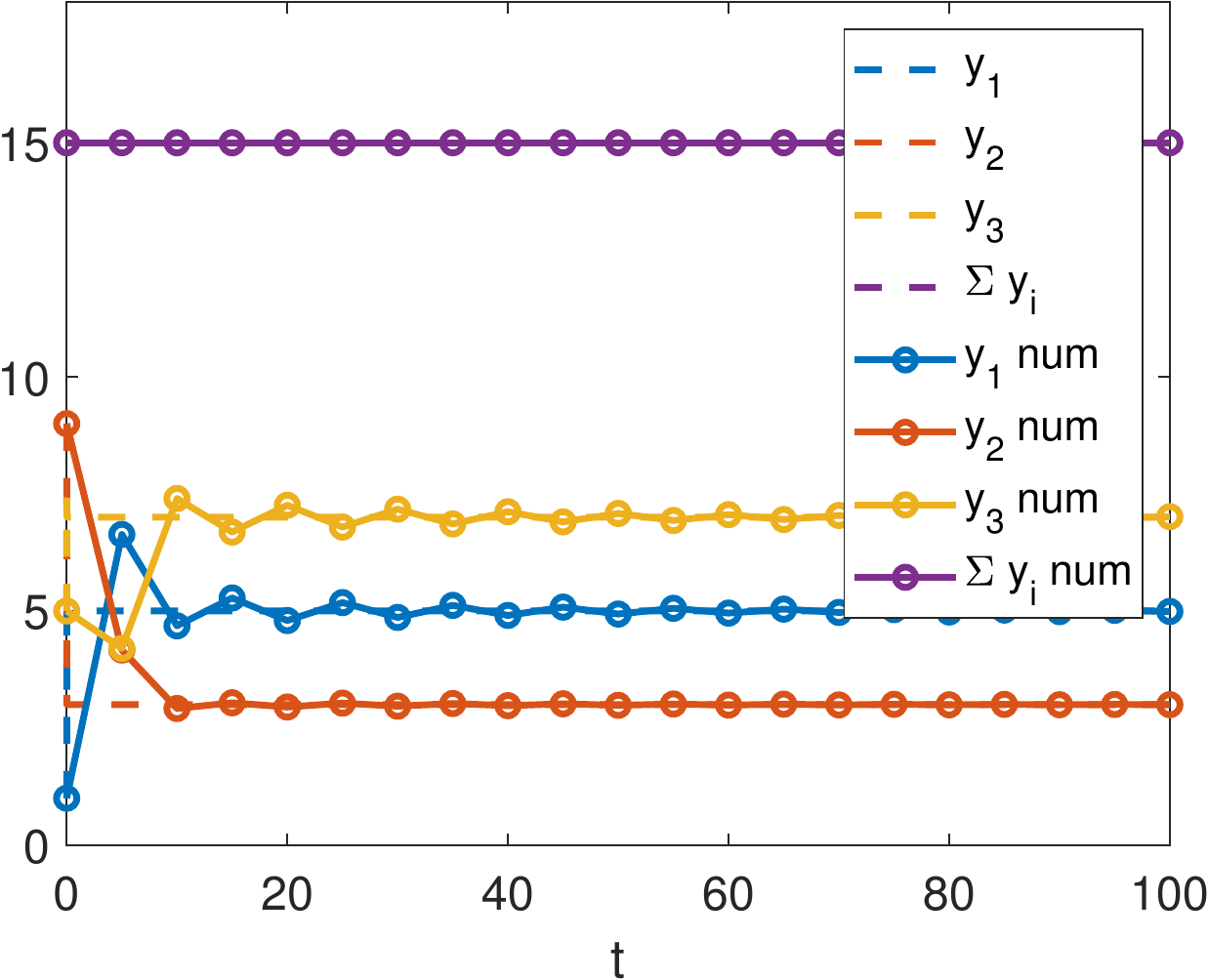}
		\subcaption{Approximation of \eqref{eq:initProbReal}}
	\end{subfigure}
	\begin{subfigure}[t]{0.49\textwidth}
		\includegraphics[width=\textwidth]{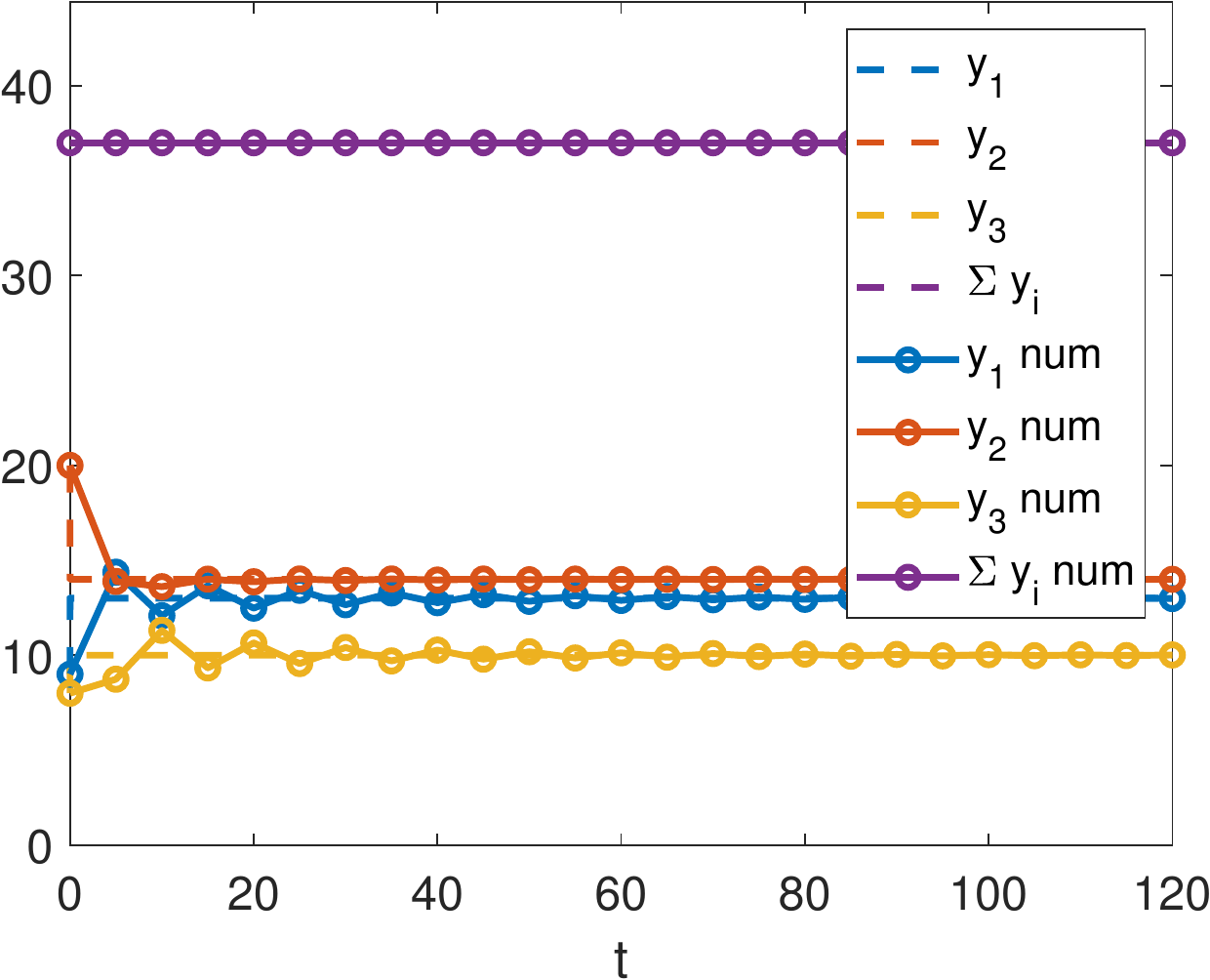}
		\subcaption{Approximation of \eqref{eq:initProbIm}}
	\end{subfigure}\\
	\begin{subfigure}[t]{0.49\textwidth}
		\includegraphics[width=\textwidth]{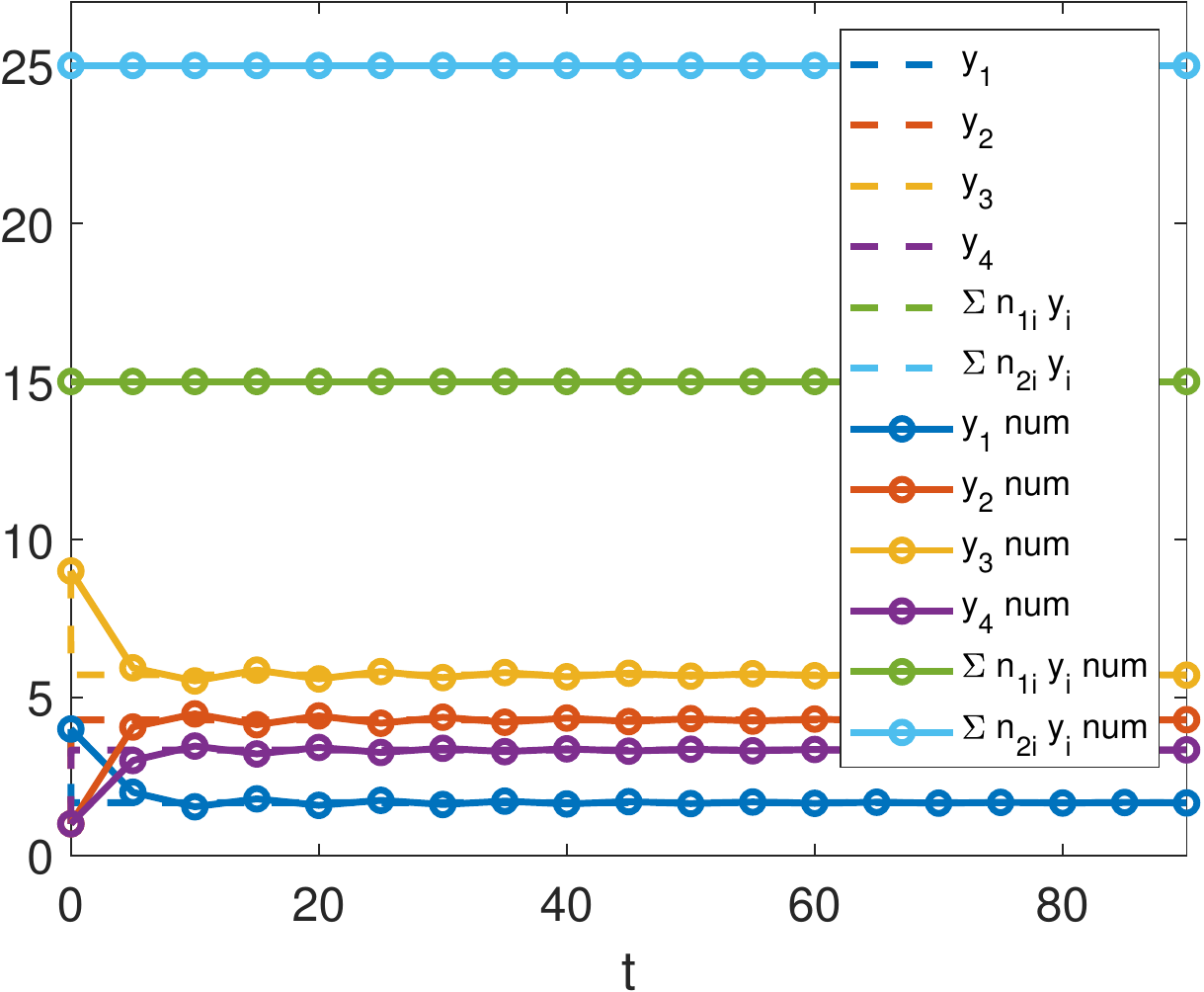}
		\subcaption{Approximation of \eqref{eq:initProb4dim}}
	\end{subfigure}
	\caption{Numerical approximations of \eqref{eq:initProbReal},\eqref{eq:initProbIm} and \eqref{eq:initProb4dim} using SSPMPRK3($\frac13$) schemes. The dashed lines indicate the exact solutions \eqref{eq:exsolReal}, \eqref{eq:exsolIm} and  \eqref{eq:exsol4dim}, where $\bn_1=\bm 1$ and $\bn_2=(1,2,2,1)^T$.}\label{Fig:SSPMPRK3initProblems}
\end{figure}
\begin{figure}[!h]
	\begin{subfigure}[t]{0.495\textwidth}
		\includegraphics[width=\textwidth]{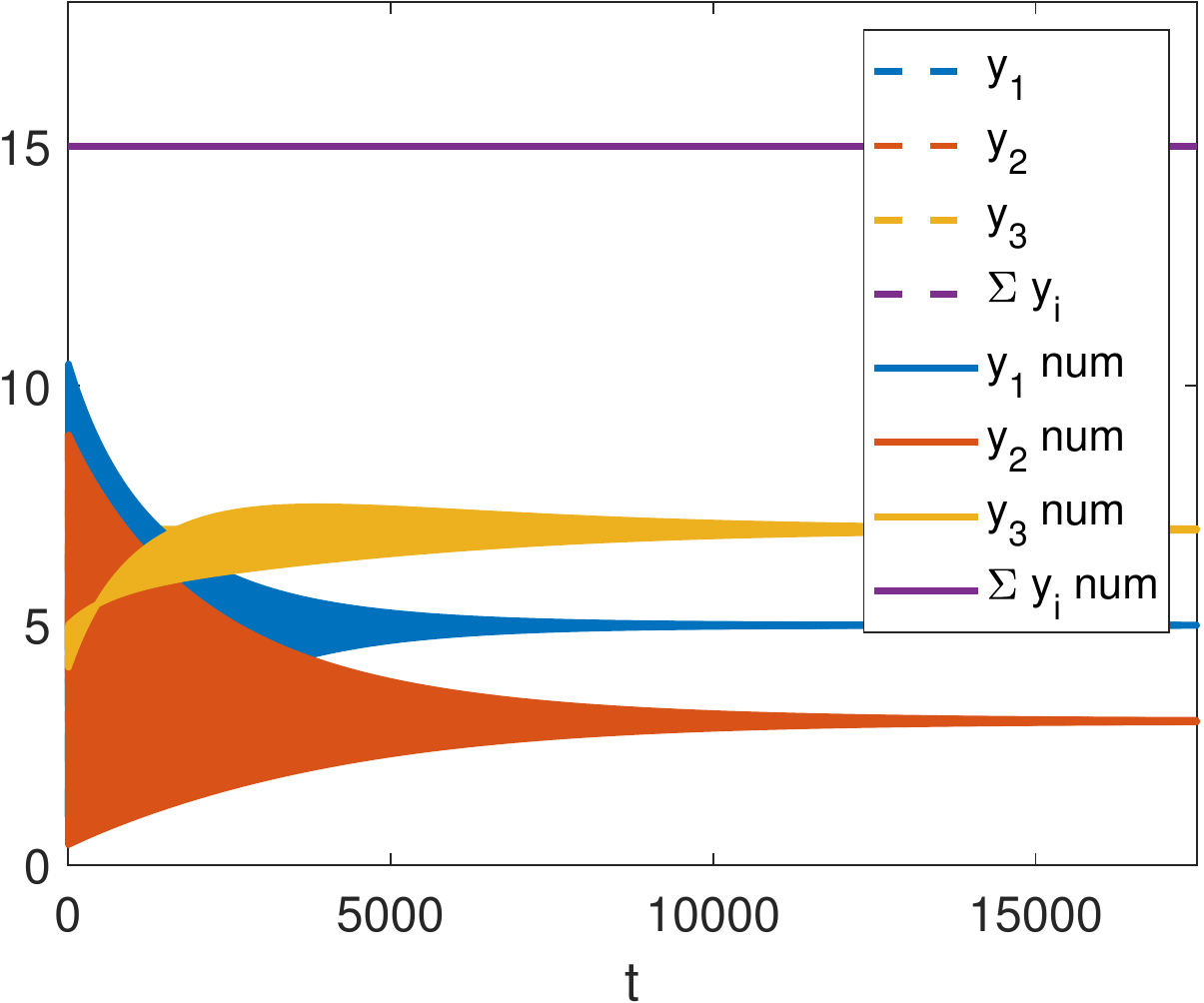}
		\subcaption{$(\alpha,\beta)=(\frac12,1)$, $\Delta t=5$}
	\end{subfigure}
	\begin{subfigure}[t]{0.495\textwidth}
		\includegraphics[width=\textwidth]{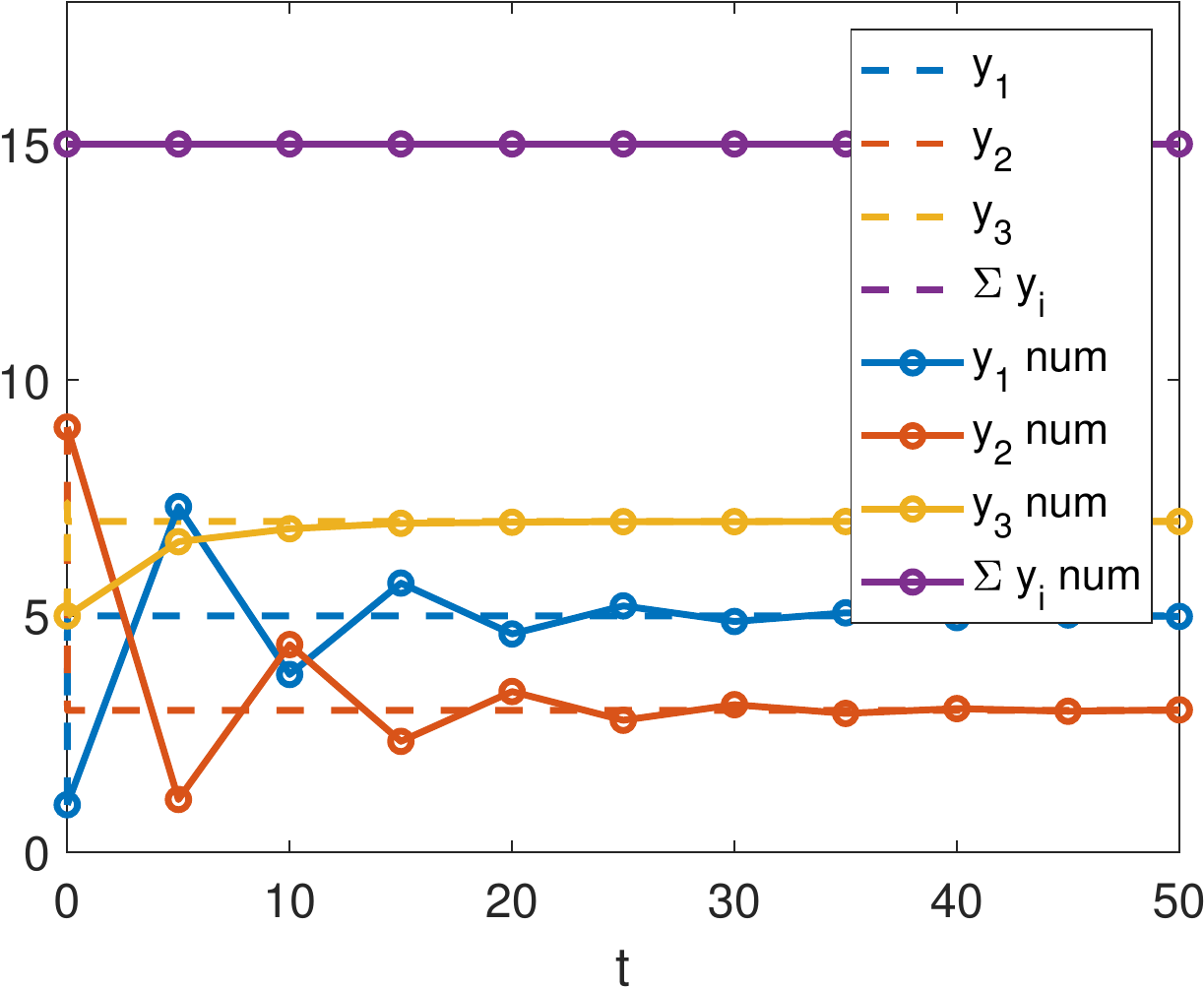}
		\subcaption{$(\alpha,\beta)=(0.1,1)$, $\Delta t=5$}
	\end{subfigure}\\
	\begin{subfigure}[t]{0.495\textwidth}
		\includegraphics[width=\textwidth]{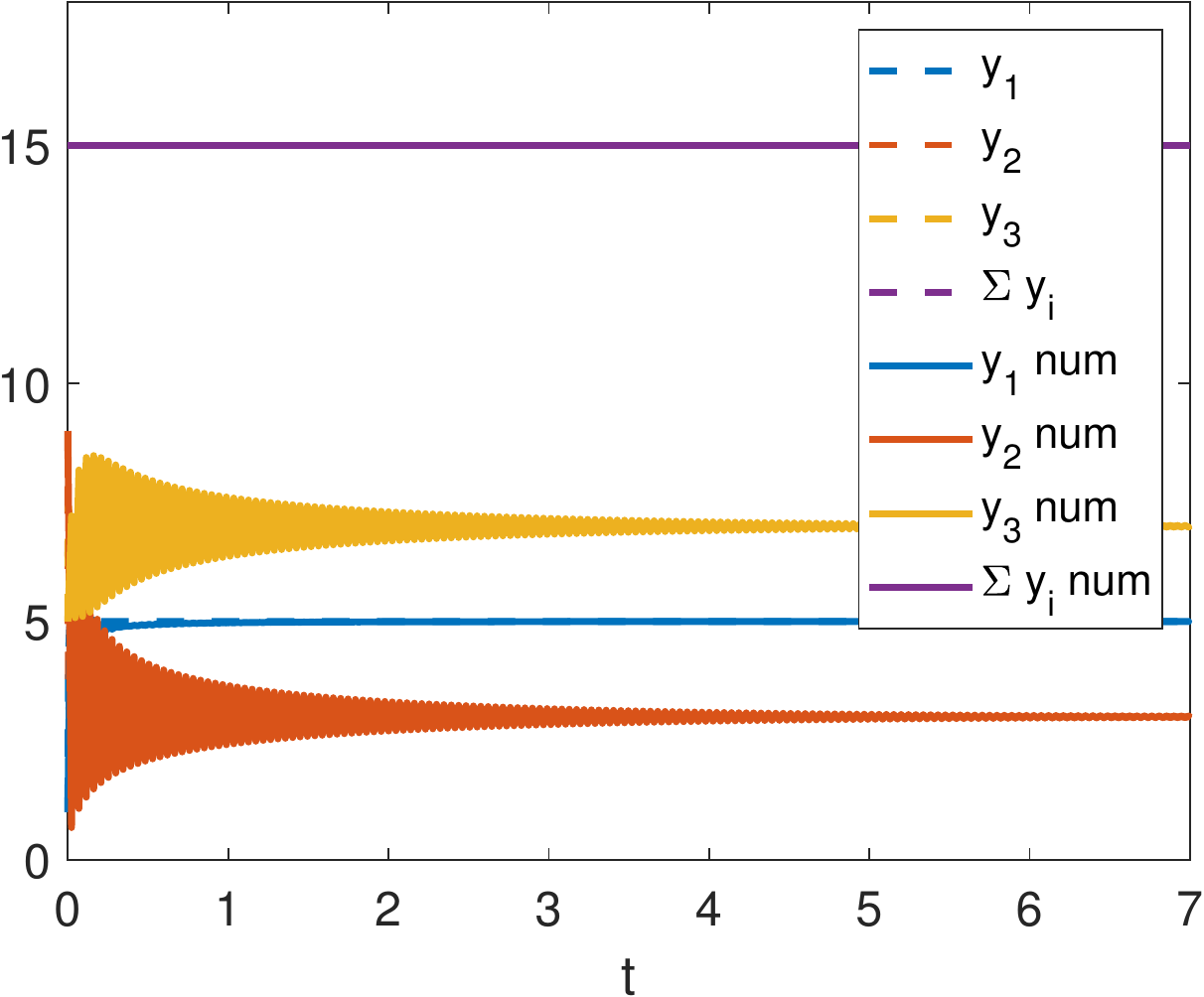}
		\subcaption{$(\alpha,\beta)=(0.2,3)$, $\Delta t\approx 0.023$}
	\end{subfigure}
	\begin{subfigure}[t]{0.495\textwidth}
		\includegraphics[width=\textwidth]{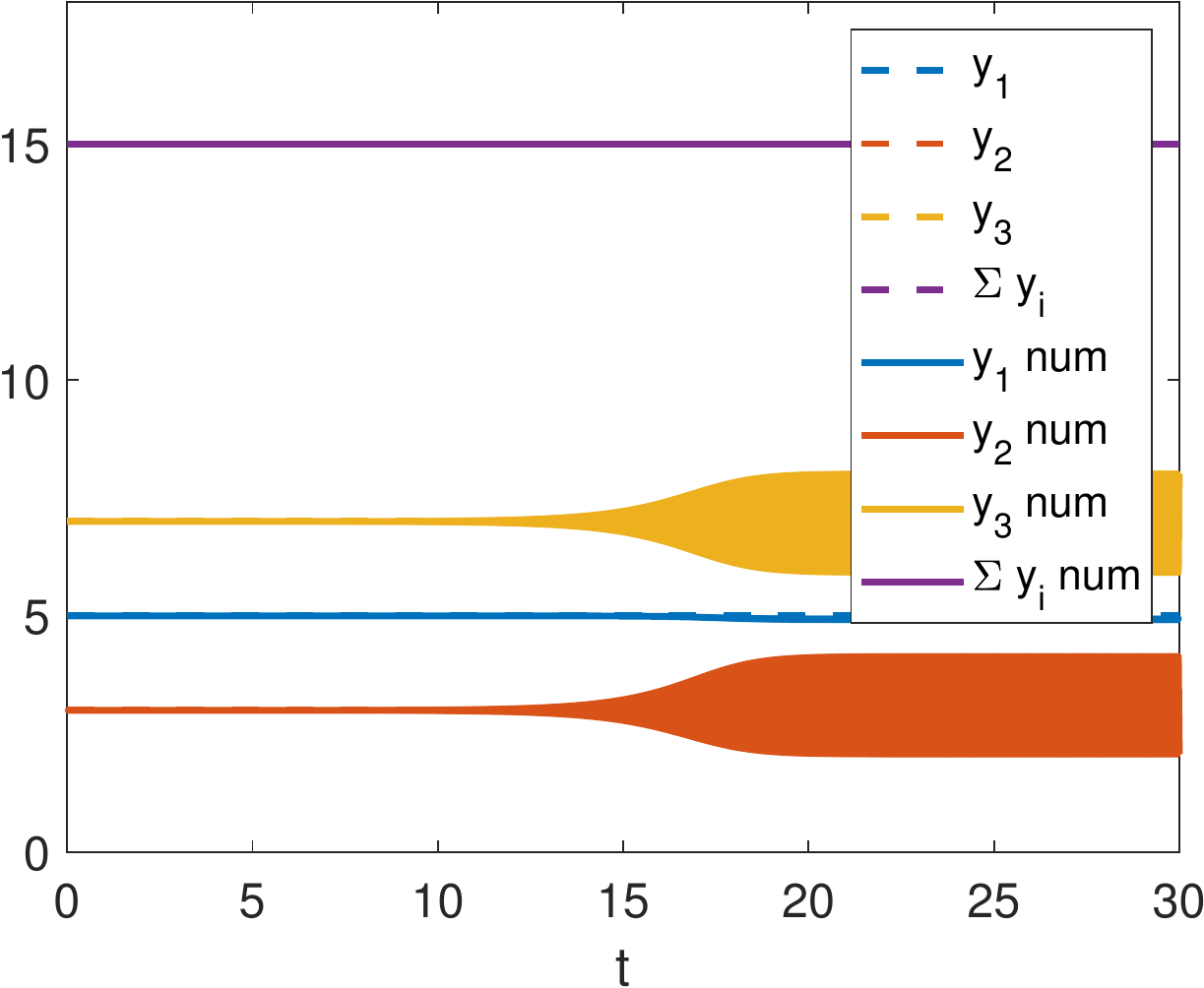}
		\subcaption{$(\alpha,\beta)=(0.2,3)$, $\Delta t\approx 0.025$}\label{Subfig:DReal}
	\end{subfigure}
	\caption{Numerical approximations of \eqref{eq:initProbReal} using the second order SSPMPRK scheme. The dashed lines indicate the exact solution \eqref{eq:exsolReal}. In \eqref{Subfig:DReal}, $\by^0=\by^*+10^{-5}(1,-2,1)^T$ is chosen. }\label{Fig:MPRKinitProbReal}
\end{figure}

\begin{figure}[!h]
	\begin{subfigure}[t]{0.495\textwidth}
		\includegraphics[width=\textwidth]{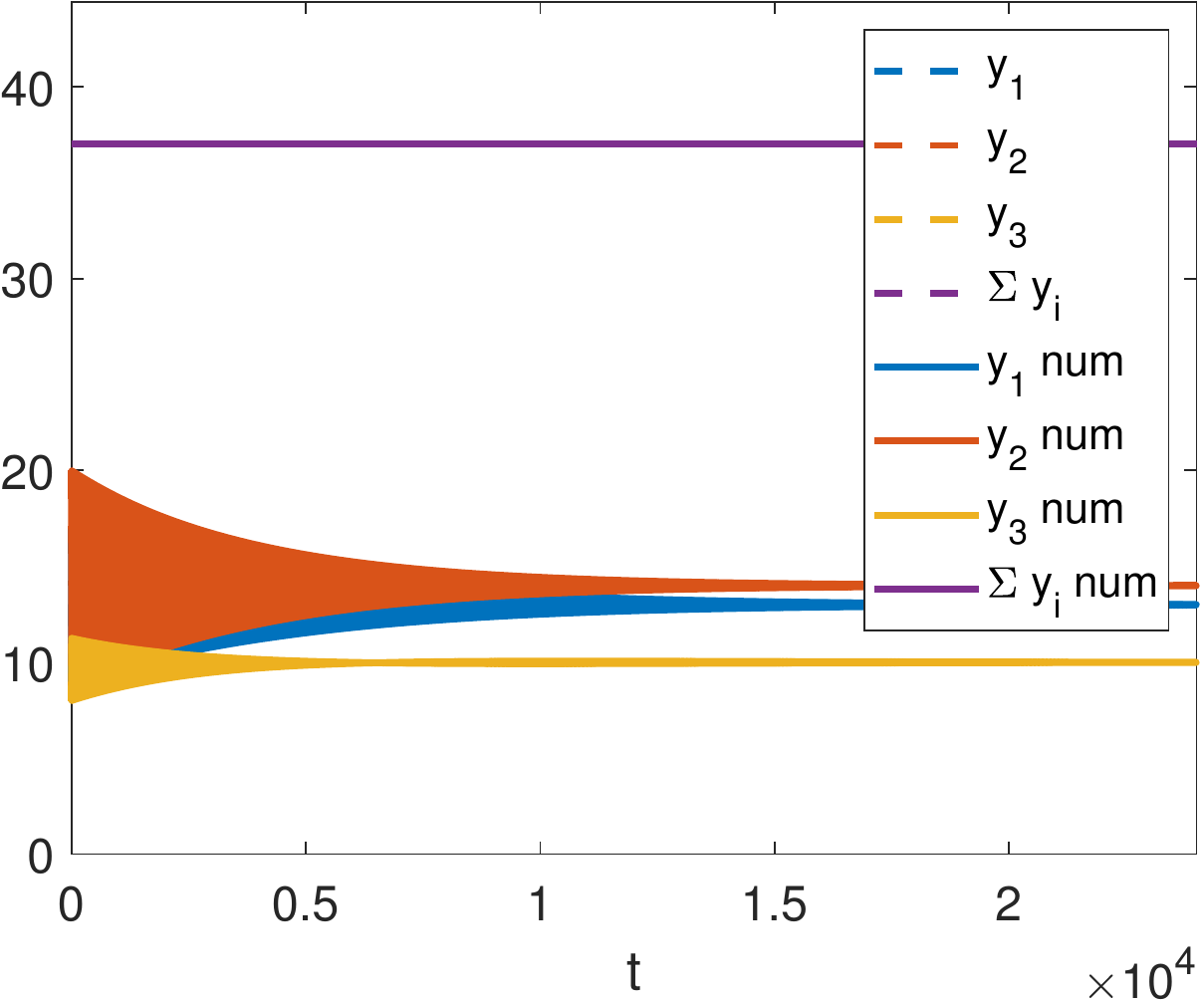}
		\subcaption{$(\alpha,\beta)=(\frac12,1)$, $\Delta t=5$}
	\end{subfigure}
	\begin{subfigure}[t]{0.495\textwidth}
		\includegraphics[width=\textwidth]{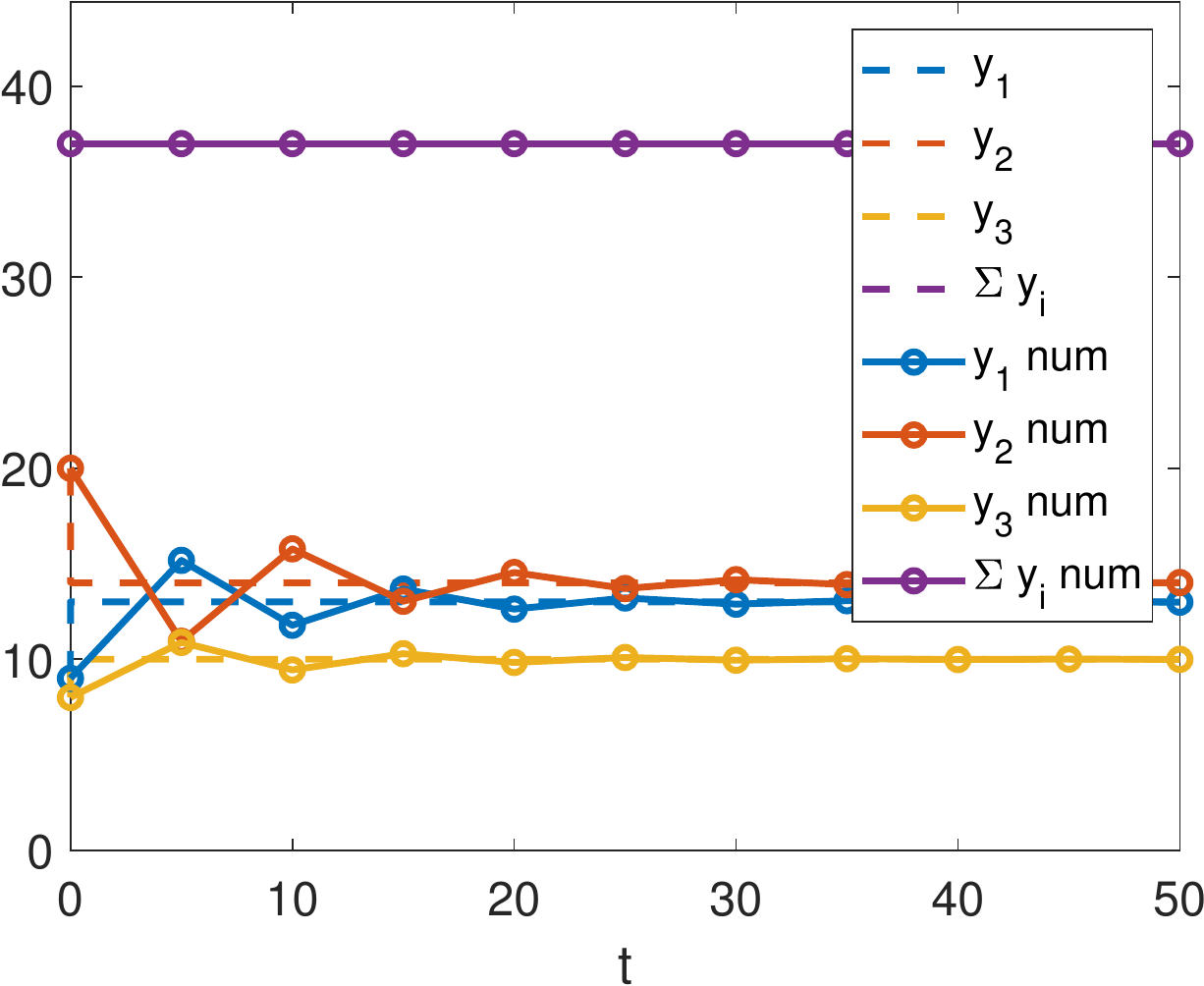}
		\subcaption{$(\alpha,\beta)=(0.1,1)$, $\Delta t=5$}
	\end{subfigure}\\
	\begin{subfigure}[t]{0.495\textwidth}
		\includegraphics[width=\textwidth]{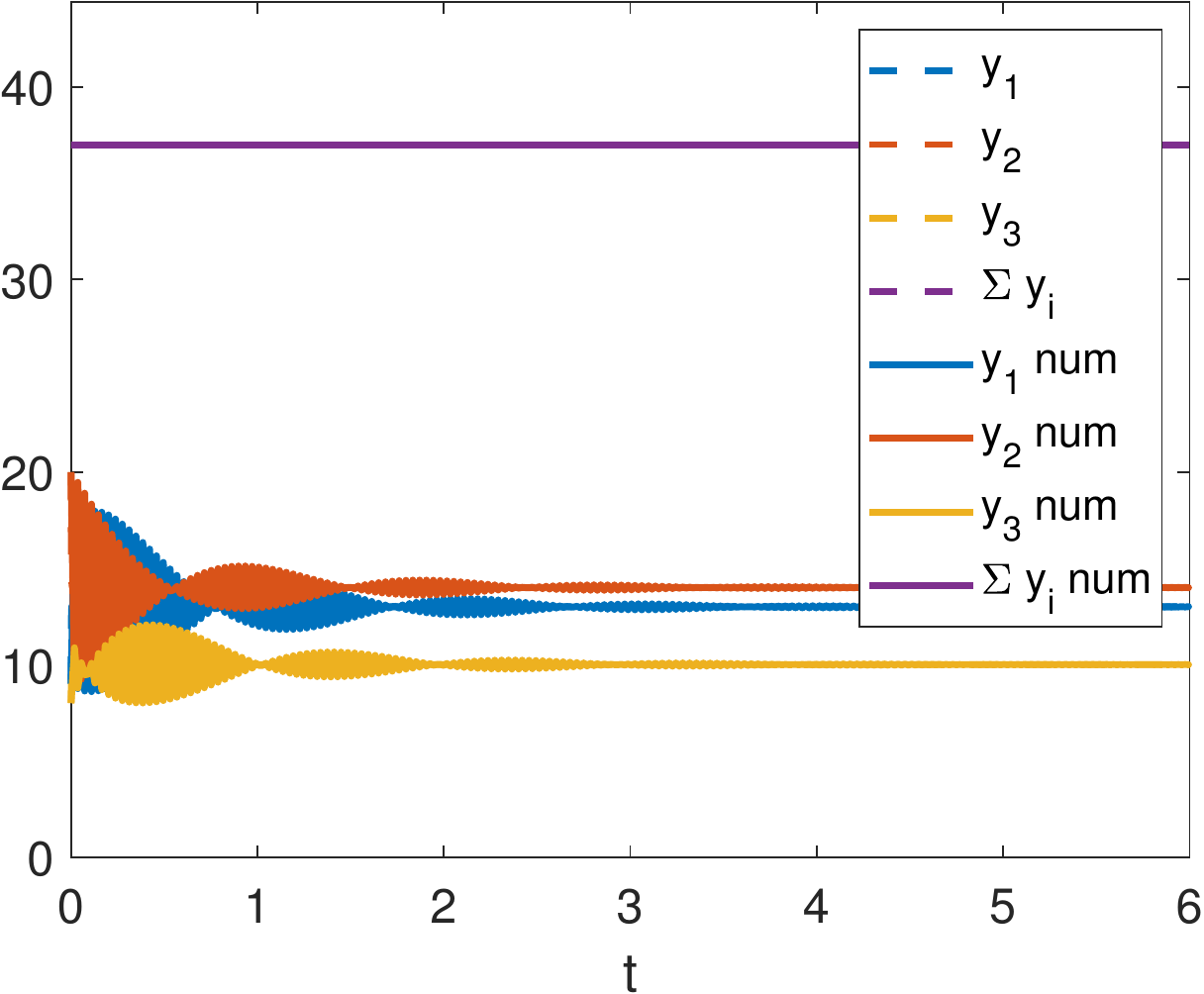}
		\subcaption{$(\alpha,\beta)=(0.2,3)$, $\Delta t\approx 0.018$}
	\end{subfigure}
	\begin{subfigure}[t]{0.495\textwidth}
		\includegraphics[width=\textwidth]{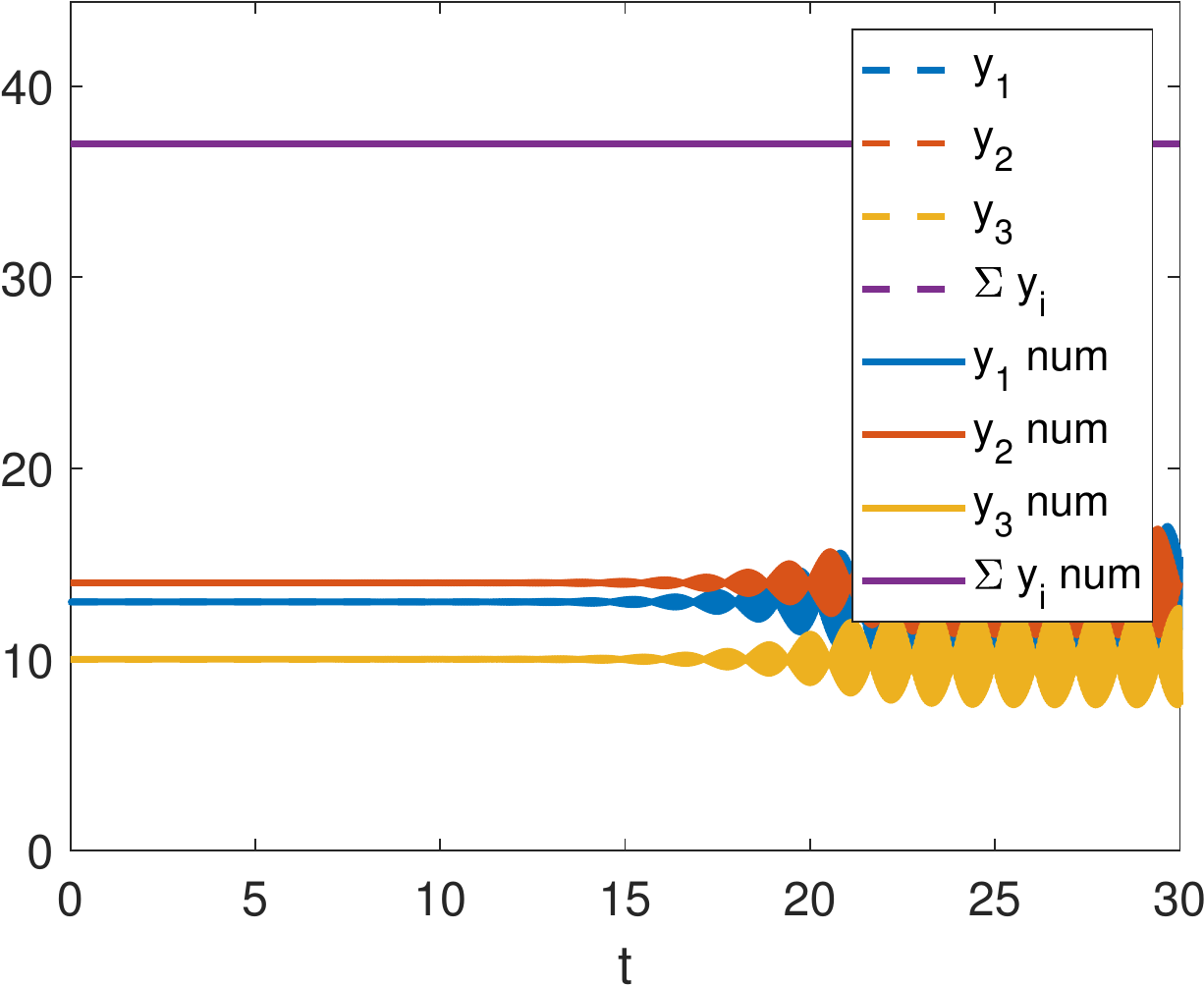}
		\subcaption{$(\alpha,\beta)=(0.2,3)$, $\Delta t\approx 0.020$}\label{Subfig:DiM}
	\end{subfigure}
	\caption{Numerical approximations of \eqref{eq:initProbIm} using  the second order SSPMPRK scheme. The dashed lines indicate the exact solution \eqref{eq:exsolIm}. In \eqref{Subfig:DiM}, $\by^0=\by^*+10^{-5}(1,-2,1)^T$ is chosen. }\label{Fig:MPRKinitProbIm}
\end{figure}

\begin{figure}[!h]
	\begin{subfigure}[t]{0.495\textwidth}
		\includegraphics[width=\textwidth]{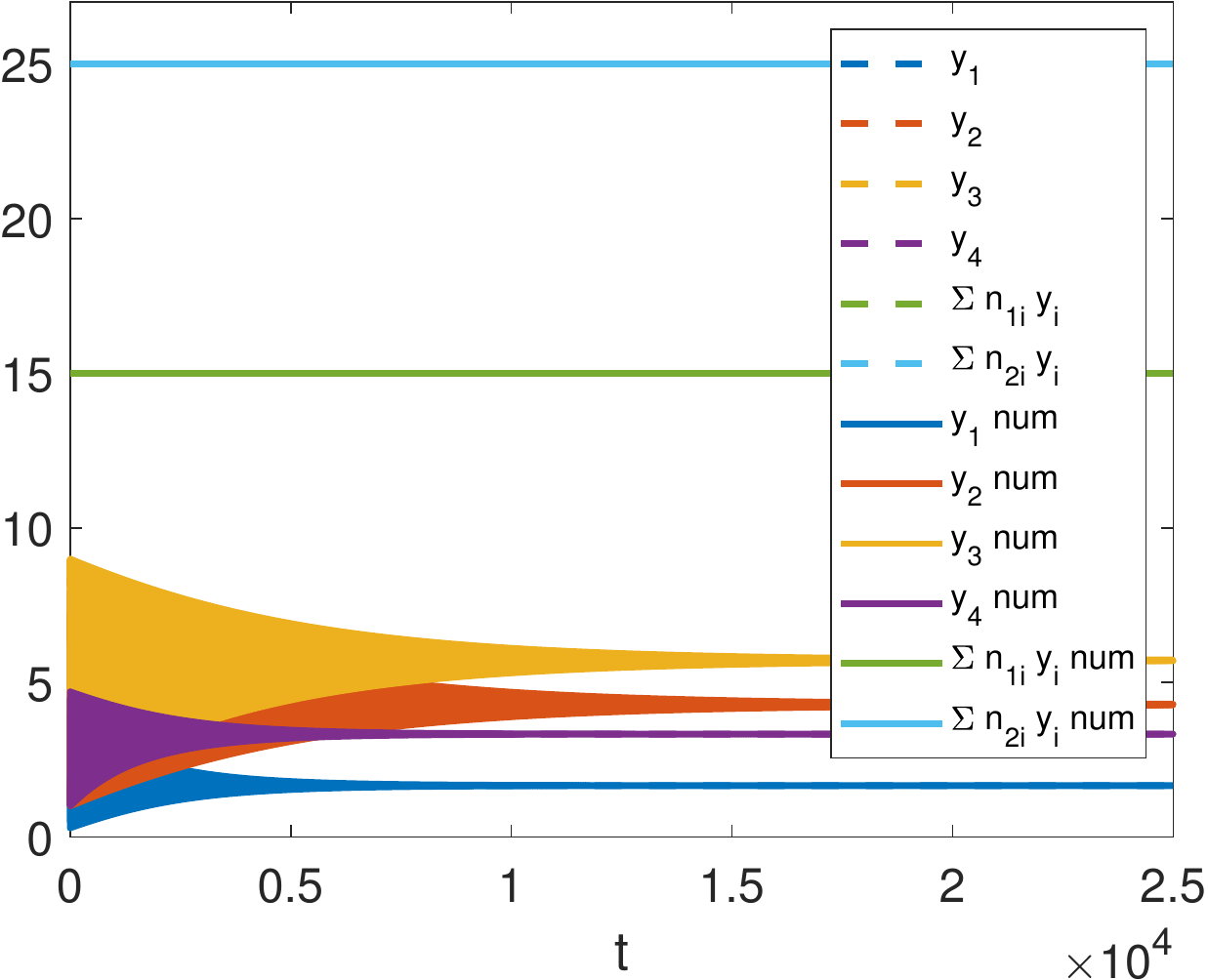}
		\subcaption{$(\alpha,\beta)=(\frac12,1)$, $\Delta t=5$}
	\end{subfigure}
	\begin{subfigure}[t]{0.495\textwidth}
		\includegraphics[width=\textwidth]{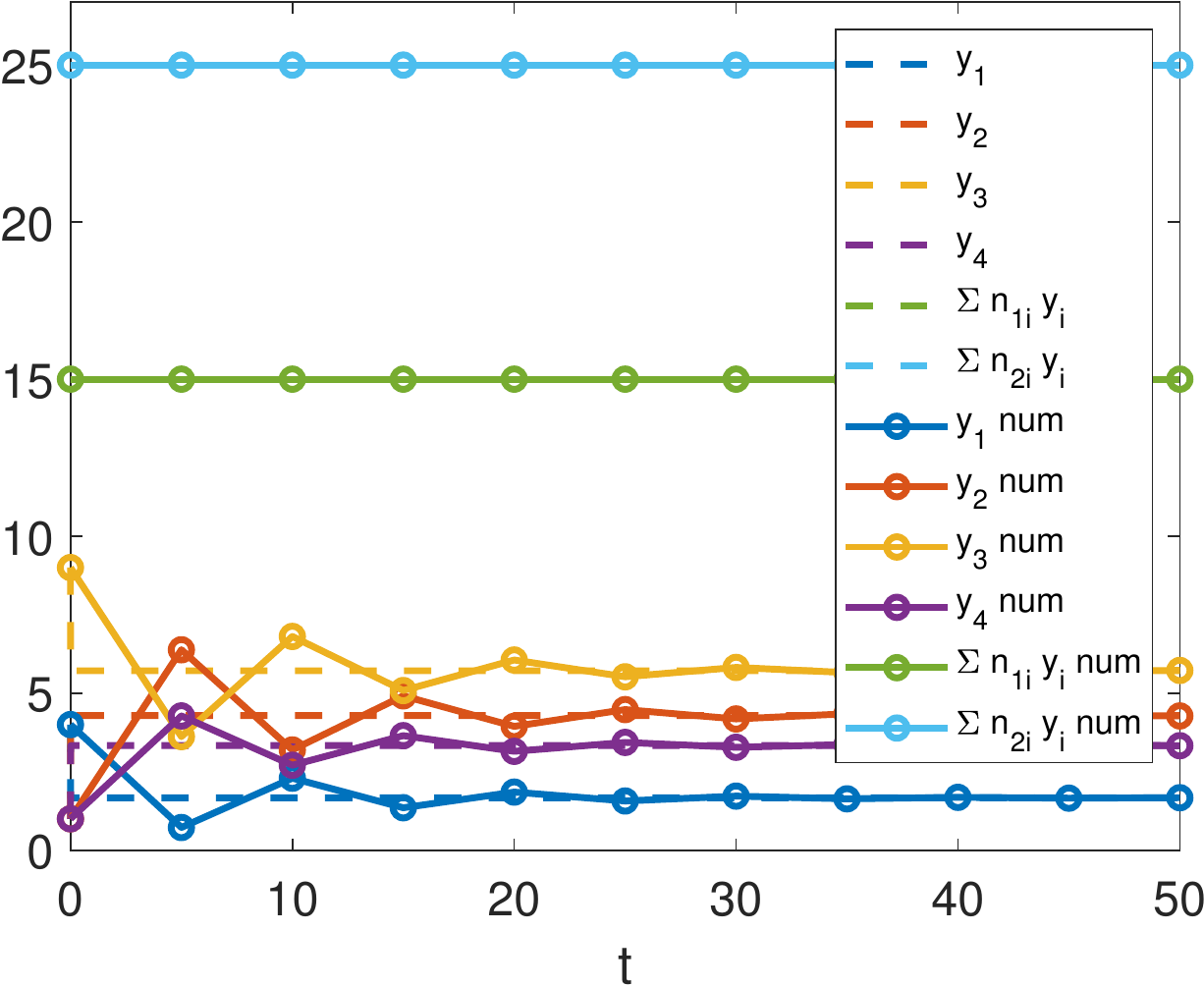}
		\subcaption{$(\alpha,\beta)=(0.1,1)$, $\Delta t=5$}
	\end{subfigure}\\
	\begin{subfigure}[t]{0.495\textwidth}
		\includegraphics[width=\textwidth]{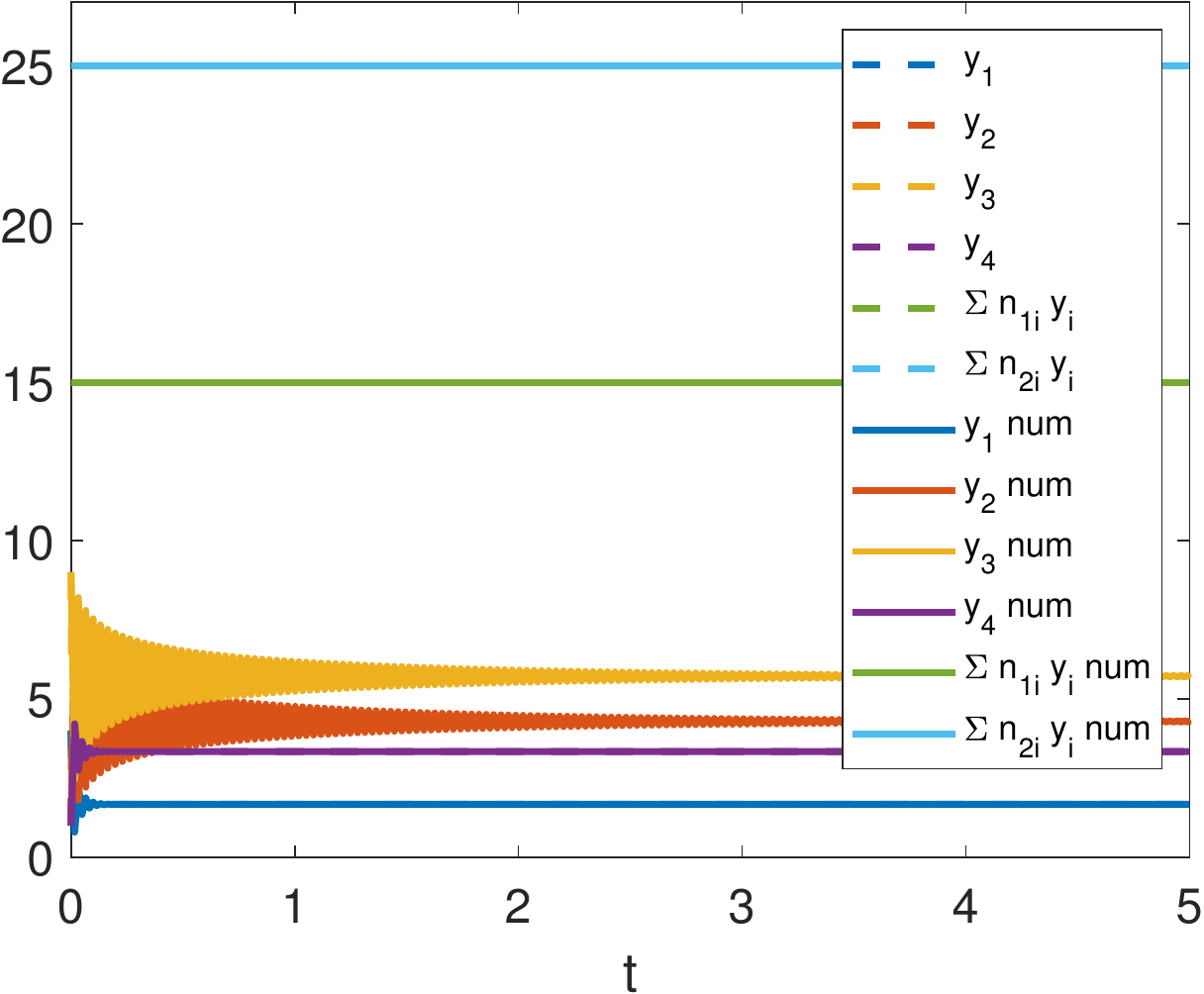}
		\subcaption{$(\alpha,\beta)=(0.2,3)$, $\Delta t\approx0.016$}
	\end{subfigure}
	\begin{subfigure}[t]{0.495\textwidth}
		\includegraphics[width=\textwidth]{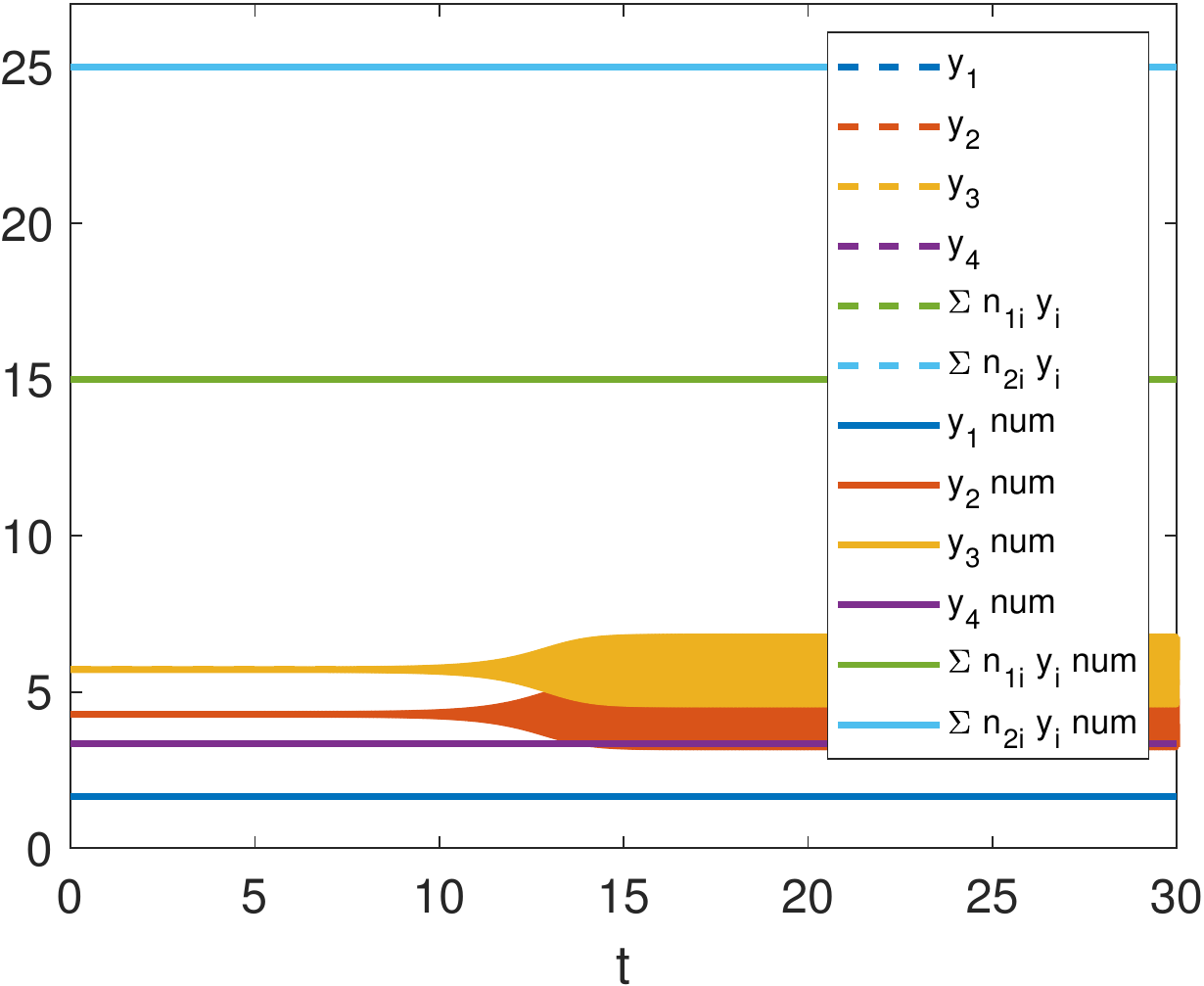}
		\subcaption{$(\alpha,\beta)=(0.2,3)$, $\Delta t\approx0.018$}\label{Subfig:D4dim}
	\end{subfigure}
	\caption{Numerical approximations of \eqref{eq:initProb4dim} using  the second order SSPMPRK scheme. The dashed lines indicate the exact solution \eqref{eq:exsol4dim}, where $\bn_1=\bm 1$ and $\bn_2=(1,2,2,1)^T$. In \eqref{Subfig:D4dim}, we used $\by^0=\by^*+10^{-5}(1,-1,1,-1)^T$. }\label{Fig:MPRKinitProb4dim}
\end{figure}
	
	\section{Summary and outlook}\label{sec:Summary}
	
We have performed stability analysis for a class of second and third order accurate
strong-stability-preserving modified Patankar Runge-Kutta (SSPMPRK) schemes
which are unconditionally positivity-preserving. This analysis allows us to identify
the range of free parameters in these SSPMPRK schemes in order to ensure stability.
Numerical experiments are provided to demonstrate the validity of the analysis. 

Here, we mention some possible future works, on applying such SSPMPRK schemes to
problems containing both convection and stiff source terms.  The convection terms
can be discretized by conservative, high resolution, essentially non-oscillatory techniques,
resulting in a very large ODE system to be discretized in time by the SSPMPRK schemes. 
The first interesting topic is to absorb the numerical fluxes from the convection terms 
into the production-destruction terms and then apply the SSPMPRK directly. Thus, the 
numerical fluxes are essentially multiplied by a factor, which may not be one, but 
should be close to one in smooth regions. The scheme should be positivity-preserving 
by design and should be high order accurate (with the worst scenario of losing at 
most one order because of the division by the spatial mesh size to the flux 
differences), but its effect on shock resolutions should be carefully assessed numerically 
and compared with the approach in \cite{MR3934688,MR3969000} in which the convection 
terms were treated by the standard high resolution schemes with SSP RK.
The second future work would be the extension of the stability analysis in this paper 
to the semi-discrete schemes arising from the multispecies reactive Euler equations. The 
difficulty is the increased complexity when the size of the ODE systems gets larger with 
spatial mesh refinements. These topics constitute our ongoing work.

	\section{Acknowledgements}
	The author Th.\ Izgin gratefully acknowledges the financial support by the Deutsche Forschungsgemeinschaft (DFG) through grant ME 1889/10-1. C.-W. Shu acknowledges support by NSF grant DMS-2010107 and AFOSR grant FA9550-20-1-0055.
	
	\bibliographystyle{plain} 
	\bibliography{cas-refs}

%
%

\end{document}